\def\mytitle{Classical and strong convexity of sublevel sets and application to attainable sets of nonlinear systems}
\def\myshorttitle{Convexity of sublevel sets and application}
\def\confidentialstring{%
This is the accepted version of a paper to be published in SIAM J. Control Optim., 2014.}
\let\headnote=\relax
\def\mykeywords{Sublevel set, convexity, strong convexity, local quadratic support, 
attainable set%
}

\def\citeTheorem{Th.}
\def\citeCorollary{Cor.}
\def\citeProposition{Prop.}
\def\citeChapter{Ch.}
\def\citeEquation{Eq.}
\def\citeDefinition{Def.}
\def\citeSection{Sec.}
\def\citeLemma{Lemma}

\def\myvspace{\vspace{0pt}}

\documentclass[final,leqno]{siamltex}
\makeatletter
\let\proof\@undefined
\let\endproof\@undefined
\newcommand{\openbox}{
\vbox{\hrule height0.6pt\hbox{%
   \vrule height1.3ex width0.6pt\hskip0.8ex
   \vrule width0.6pt}\hrule height0.6pt
  }
}
\providecommand{\qedsymbol}{\openbox}%
\newenvironment{proof}[1][Proof]{
 \normalfont \topsep6\p@\@plus6\p@\relax
  \trivlist
 \item[
\hskip\labelsep
       \itshape
   \hspace{15pt}#1\@addpunct{.}]\ignorespaces
}{\qquad \qedsymbol
}
\makeatother
\makeatletter
\newenvironment{example}{
    \refstepcounter{theorem}%
  \par\textit{Example \arabic{section}.\arabic{theorem}\@addpunct{.}\ignorespaces }%
}{\par}%
\makeatother
\makeatletter
\newcounter{hypocnt}
\newenvironment{hypothesis}[1]{
    \refstepcounter{hypocnt}%
\vspace{5pt}
~\\\ensuremath{(H_\arabic{hypocnt})}\ignorespaces \quad #1
}{\vspace{5pt}\par}
\makeatother
\newcommand{\reftohypothesis}[1]{\ensuremath{(H_{\textnormal{#1}})}}
\usepackage{gunther2esiam}
\usepackage{cite}
\usepackage{graphicx}
\usepackage{psfrag}
\usepackage{xcolor}
\usepackage{amssymb}
\newcommand{\textfrac}[2]{{#1}/{#2}}

\newcommand{\R}{\mathbb{R}}
\newcommand{\sequence}[3]{\ifthenelse{\equal{#2}{~}}%
                                     {({#1}_k)_{k\in \mathbb{N}}}%
                                     {({#1}_k)_{k\in \mathbb{N}} \text{ in } {#2}}}%
\title{\MakeUppercase{\mytitle}\footnotemark[1]~\footnotemark[3]
}

\author{Alexander Weber\footnotemark[2]\and Gunther Reissig\footnotemark[2]}

\ifx
\hypersetup\undefined\def\href#1#2{\texttt{#2}}
\else
\hypersetup{
pdftitle={\mytitle},
pdfauthor={Alexander Weber, Gunther Reissig, A.Weber@unibw.de, http://www.reiszig.de/gunther/},%
pdfsubject={\confidentialstring{} \headnote},
pdfkeywords={\mykeywords}
}
\fi
\begin{document}
\makeatletter
\let\figurename\@undefined
\newcommand\figurename{Fig.}
\makeatother
\maketitle
\renewcommand{\thefootnote}{\fnsymbol{footnote}}
\footnotetext[1]{This is the accepted version of a paper to be published in \textit{SIAM J. Control Optim.}, 2014.}
\footnotetext[3]{This work has been supported by the German Research Foundation (DFG) under grant no. RE 1249/3-1.%
}

\footnotetext[2]{University of the Federal Armed Forces Munich,
Department of Aerospace Eng., Institute of Control Eng. (LRT-15), D-85577 Neubiberg
(Munich), Germany, \href{mailto:A.Weber@unibw.de}{\textcolor{black}{A.Weber@unibw.de}},\linebreak \url{http://www.reiszig.de/gunther/}}
\renewcommand{\thefootnote}{\arabic{footnote}}
\begin{abstract}
Necessary and sufficient conditions for convexity and strong convexity, respectively, of connected sublevel sets that are defined by finitely many real-valued $C^{1,1}$-maps are presented. A novel characterization of strongly convex sets in terms of the so-called local quadratic support is proved. The results concerning strong convexity are used to derive sufficient conditions for attainable sets of 
continuous-time nonlinear systems to be strongly convex. An application of these conditions is a novel method to over-approximate attainable sets when strong convexity is present. 
\end{abstract}
\pagestyle{myheadings}
\thispagestyle{plain}
\markboth{\uppercase{Alexander Weber and Gunther Reissig}}{\MakeUppercase{\myshorttitle}}
\section{Introduction}
\label{s:introduction}
In this paper we investigate necessary and sufficient conditions for connected sublevel sets of the form
\begin{equation}
\label{e:sublevelset}
\{ x \in U \, | \, g_1(x) \leq 0,\ldots, g_m(x) \leq 0\}
\end{equation}
to be convex and strongly convex, respectively. The ingredients of \ref{e:sublevelset} are an open subset $U$ of the $n$-dimensional real space, a positive integer $m$, and real $C^{1,1}$-functions $g_1,\ldots,g_m$ (continuous with Lipschitz continuous derivative) with domain $U$. We focus on conditions that are given explicitly in terms of properties of $g_1,\ldots,g_m$. We will review the concept of strong convexity later in the introduction and proceed with introductory remarks on ordinary convexity of sublevel sets of the form \ref{e:sublevelset}.

We give two motivations for investigating convexity of sublevel sets. The first arises from convex optimization as follows. In optimization problems with inequality constraints the feasible set usually takes the form \ref{e:sublevelset}, and such a problem may be considered having a linear objective by straightforward transformation. Assuming convexity of the feasible set has therefore the nice, widely known consequence that every local solution is a global one. Moreover, simple algorithms for constrained optimization successfully find global solutions. The second motivation arises from results on so-called semidefinite representability of sets of the form \ref{e:sublevelset}, e.g. \cite{Lasserre08,Lasserre09,HeltonNie09,HeltonNie10}. These results assume convexity of \ref{e:sublevelset} but do not provide conditions to verify this hypothesis. Thus, the problem of verifying convexity of \ref{e:sublevelset} is of practical relevance. Next, we discuss the known conditions for convexity of \ref{e:sublevelset} that depend explicitly on properties of $g_1,\ldots,g_m$. 

A well-known sufficient condition for the set \ref{e:sublevelset} to be convex is that all functions $g_1,\ldots,g_m$ are convex functions but this condition is far from being necessary. 
In the case where $g_1,\ldots,g_m$ are polynomials, criteria for convexity of \ref{e:sublevelset} have been presented in \cite{Lasserre08,Lasserre10,HenrionLouembet12} with the result in \cite{HenrionLouembet12} being incorrect as Examples \ref{ex:thm:convex:1} and \ref{ex:thm:convex:2} of the present paper reveal. In \cite{Lasserre08,Lasserre10} the criteria are given in terms of so-called certificates, i.e., the existence of polynomials that satisfy a certain relation implies convexity of \ref{e:sublevelset} and conversely. The verification of this condition requires solving semidefinite programs of large size, in general. So, these certificates are suitable to verify convexity only numerically. A criterion for the non-polynomial case is \cite[\citeLemma~4.3]{Lasserre08} where differentiability of the functions involved is required (and some other convenient assumptions). This characterization is based on a \textit{global} condition of \textit{first} order in the following sense: For any fixed boundary point of \ref{e:sublevelset} one has to verify a relation between the boundary point and every other point of \ref{e:sublevelset}, and in this relation the first derivatives $g_1',\ldots,g_m'$ of $g_1, \ldots, g_m$ are involved. 
In contrast, a criterion given in \cite{i07Convex} for the special case $m = 1$ is a purely \textit{local} condition of \textit{second} order: For any fixed boundary point of \ref{e:sublevelset} one needs to verify properties on first and (generalized) second-order derivatives at that point. 

We intend to extend the result of \cite{i07Convex} to the case $m>1$. On the one hand, our novel results consist of second-order conditions as described above, and therefore require slightly more smoothness (Lipschitz continuity of $g_1',\ldots,g_m'$) than \cite[\citeLemma~4.3]{Lasserre08}. On the other hand, our conditions are purely local, so they are easier to be verified in practice.

Geometrical properties of sets are of significant importance in control theory, e.g. \cite{Plis74,FattoriniFrankowska90,FattoriniFrankowska90b,CannarsaFrankowska06,i07Convex}. An example is the concept of strong convexity, introduced in \cite{Mayer35} and later renamed \cite{LevitinPolyak66}, which we formally define in the subsequent paragraph. The concept is analogous to and implies ordinary convexity, while the converse is not true if the dimension of the space exceeds $1$.
Several results in both optimization and control theory, e.g. \cite{Veliov89,JourneeNesterovRichtarikSepulchre08,i11abs}, rely on the hypothesis that certain sets are strongly convex rather than merely convex. Hence, detecting strong convexity is also of interest. Our particular motivation to investigate strong convexity of sublevel sets is to increase the efficiency of a method presented in \cite{i11abs} for over-approximating attainable sets of nonlinear dynamical systems which is an essential issue in the so-called \begriff{abstraction based controller design} \cite{GrueneJunge07,i11abs,ReissigRungger13,Tabuada09,RunggerMazoTabuada13}. The efficiency of this approach depends, among others, on the quality of the over-approximation method used. 
To give more details on how to use strong convexity and sublevel sets, and on how to improve the over-approximation of attainable sets, we briefly discuss the method in \cite{i11abs} after having introduced strong convexity more formally.\looseness=-1

For the $n$-dimensional real space $\R^n$ endowed with the Euclidean inner product $\innerProd{\cdot}{\cdot}$ and corresponding norm $\|\cdot \|$, strong convexity is defined as below. Here and throughout, $\cBall(c,r)$ denotes the closed ball
in $\R^n$ of radius $r$ centered at $c\in \R^n$, where the convention $\cBall(c,0) = \{c\}$ is adopted.
\begin{definition}
\label{def:StrongConvexity}
Let $r>0$. A set $\Omega \subseteq \R^n$ is called \begriff{$r$-convex} if 
\begin{equation}
\bigcap_{x,y \in \cBall(c,r)} \cBall(c,r) \subseteq \Omega
\label{linse}
\end{equation}
for all $x,y \in \Omega$. The intersection in \ref{linse} is taken over all balls $\cBall(c,r)$
containing $x$ and $y$ and is considered equal to $\R^n$ if no such
ball exists. $\Omega$ is called \begriff{strongly convex} if $\Omega$ is
$s$-convex for some $s>0$. 
\end{definition}

Definition \ref{def:StrongConvexity} makes the connection to ordinary convexity obvious: The role of line segments is now assumed by so-called \begriff{lenses} \cite{Blanc43} by which we mean the set on the left hand side of \ref{linse}. Moreover, the role of supporting half-spaces is replaced by supporting \textit{balls} as the following proposition shows \cite{FrankowskaOlech81,i12sConvex}. 
\begin{proposition}
\newcommand{\unitnormal}{v}
\label{p:globalchar}
Let $r > 0$ and $\Omega \subseteq \R^n$ be closed. Denote by $\boundary \Omega$ the boundary of $\Omega$. Then the following conditions are equivalent: 
\begin{asparaenum}[(i)]
\item
\label{p:globalchar:i}
{$\Omega$ is $r$-convex.}
\item
\label{p:globalchar:ii}
{$\Omega$ is convex and $\Omega \subseteq \cBall(x - r \unitnormal,r)$
for all $x \in \partial \Omega$ and all $\unitnormal\in \R^n$ such that $\| \unitnormal\| = 1$ and $\innerProd{\unitnormal}{y-x}\leq 0$ for all $y \in \Omega$.}
\end{asparaenum}
\end{proposition}
\myvspace
We continue with the overview of the over-approximation method presented in \cite{i11abs}. For simplicity, we consider an autonomous plant whose dynamics are given by 
\begin{equation}
\label{e:ode}
\dot x = F(x)
\end{equation}
with right hand side $F\colon \R^n \to \R^n$ of class $C^{1,1}$. Denote by $\varphi\colon \R_+ \times \R^n \to \R^n$ the flow of \ref{e:ode}: $\varphi(0,x_0)=x_0$ for $x_0 \in \R^n$, and $\varphi'(t,x_0)= F(\varphi(t,x_0))$ for all non-negative $t$ where the derivative is taken with respect to $t$. In \cite{i11abs}, sufficient conditions for the attainable set $\varphi(t,\Omega_0)$ to be convex are given in terms of $\Omega_0 \subseteq \R^n$ and $t>0$. Hence, an over-approximation with a finite number of supporting half-spaces of $\varphi(t,\Omega_0)$ is ensured by the well-known support property of convex sets. The required pairs of point and normal vector defining the half-spaces are obtained from the adjoint equation to \ref{e:ode}, e.g. \cite{i11abs}. We remark that if $\Omega_0$ is of the form \ref{e:sublevelset} then so is $\varphi(t,\Omega_0)$. This fact and our results on strong convexity of sublevel sets for $m = 1$ will imply sufficient conditions in terms of $\Omega_0$ and $t$ for $\varphi(t,\Omega_0)$ to be \textit{strongly} convex.  
Then we will be in a position to use property \ref{p:globalchar:ii} of Proposition \ref{p:globalchar} to approximate $\varphi(t,\Omega_0)$ by supporting balls more accurately than by the same number of supporting half-spaces. We emphasize that no more data than for the corresponding half-spaces will be needed, namely the same pairs of point and normal vector defining the half-spaces. See \ref{fig:SuppBalls}. (The radius of the supporting balls can be obtained from properties of the right hand side $F$ or the flow of \ref{e:ode}; see Section \ref{s:application}.)\looseness=-1

\begin{figure}
\newcommand{\unitnormal}{v}
\begin{center}
  \psfrag{O}[t][t]{$\Omega$}
\psfrag{H}[][]{$\widehat\Omega$}
\psfrag{1}[][]{\small$x_1$}
\psfrag{2}[][]{\small$x_2$}
\psfrag{3}[][]{\small$x_3$}
\psfrag{4}[][]{\small$x_4$}
\psfrag{5}[][]{\small$\unitnormal_1$}
\psfrag{6}[][]{\small$\unitnormal_2$}
\psfrag{7}[][]{\small$\unitnormal_3$}
\psfrag{8}[][]{\small$\unitnormal_4$}
  \includegraphics[width=.5\linewidth]{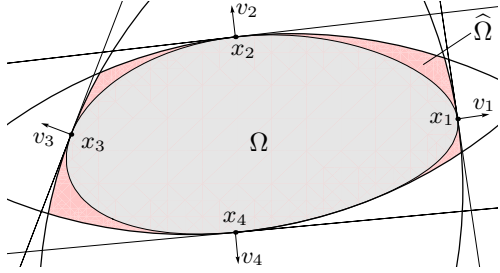}
\end{center}
\caption{\label{fig:SuppBalls}%
Over-approximation of the $r$-convex set $\Omega$ by the
intersection $\widehat \Omega$ of four supporting balls $\cBall(x_i-r\unitnormal_i,r)$, $i=1,\ldots,4$ \textnormal{\cite{i13qsuppc}}. The intersection of the four
supporting half-spaces that are determined by the same normal vectors $\unitnormal_i$ is a
less accurate approximation of $\Omega$.
}
\vspace*{-\baselineskip}%
\end{figure}

In analogy to ordinary convexity, the necessary and sufficient conditions for strong convexity of sublevel sets, which we will present, are purely local and of second order. Our results will recover as a special case a sufficient condition given in \cite{JourneeNesterovRichtarikSepulchre08} for $m = 1$. The result in \cite{JourneeNesterovRichtarikSepulchre08} is, to the best of our knowledge, the only previously known condition that depends explicitly on properties of $g_1,\ldots,g_m$.

The rest of the paper is organized as follows. In Section \ref{s:preliminaries} we set up notation and terminology. The main theorems on convexity and strong convexity, respectively, of sublevel sets are stated and proved in Section \ref{s:mainresultssublevelsets}. Two brief examples will demonstrate that these results may be useful for various applications.  The proof of the main theorem on strong convexity involves exploiting the so-called quadratic support property which is therefore introduced. Moreover, we present a proposition that gives a bound for generalized second-order derivatives of maximum functions. The proof of the main theorem for ordinary convexity is based, among other things, on that result. In Section \ref{s:application} we investigate sufficient conditions for attainable sets of nonlinear systems of the form \ref{e:ode} to be strongly convex as an application of the main 
theorem on strong convexity of sublevel sets. Our results are demonstrated on attainable sets that are relevant for the computation of a discrete abstraction (including past information \cite{GrueneMueller08,i11abs}) of a system consisting of a pendulum mounted on a cart.

We remark that some of the presented results have been announced in \cite{i13qsuppc}.
\section{Basic notation and terminology}
\label{s:preliminaries}
Throughout the paper, we mean by (ordinary) {convexity} its classical definition, e.g. \cite[\citeDefinition~1.3]{Valentine64}. $\mathbb{R}$ and $\mathbb{Z}$ denote the sets of real numbers and
integers, respectively, $\mathbb{R}_{+}$ and $\mathbb{Z}_{+}$, their
subsets of non-negative elements,
and $\mathbb{N} = \mathbb{Z}_{+} \setminus \{ 0 \}$.
$\intcc{a,b}$, $\intoo{a,b}$,
$\intco{a,b}$, and $\intoc{a,b}$
denote closed, open and half-open, respectively,
intervals with end points $a$ and $b$. As already mentioned, $\innerProd{\cdot}{\cdot}$ and $\|\cdot\|$ denote the standard
Euclidean product and norm, respectively, i.e.,
$
\innerProd{x}{y} = \sum_{i=1}^n x_i y_i
$
and
$\| x \| = \innerProd{x}{x}^{1/2}$
for any $x, y \in \mathbb{R}^n$.
$\oBall(x,r)$ and $\cBall(x,r)$ denote the
open and closed, respectively, ball of radius $r > 0$ centered at
$x$. 
The closure, the interior, the boundary and the convex hull of a set
$\Omega \subseteq \mathbb{R}^n$ are denoted by $\closure\Omega$,
$\interior\Omega$, $\boundary{\Omega}$, and $\conv \Omega$,
respectively. In particular, $\conv\{x,y\}$ is the line segment
$\{tx+(1-t)y \, | \, t \in [0,1]\}$.
A vector $v \in \mathbb{R}^n$ is
\begriff{normal to $\Omega$ at $x \in \boundary \Omega$} if
$\innerProd{v}{y-x} \leq 0$ for all $y \in \Omega$. If, additionally,
$\| v \|= 1$ then $v$ is a \begriff{unit normal to $\Omega$ at $x$}. 
Two vectors $x$ and $y$ are \begriff{perpendicular}, $x \perp y$, if $\innerProd{x}{y}=0$.
The derivative and the inverse of a
map $f$ is denoted by $f'$ and $f^{-1}$, respectively, and $f^{\ast}$
is the transpose of $f$ if $f \colon \mathbb{R}^n \to \mathbb{R}^m$ is
linear. We set $fh^k:=f(h,\ldots,h)$ if $f$ is $k$-linear. $f$ is \begriff{of class $C^k$} if $f$ is $k$-times continuously differentiable, and \begriff{of class $C^{1,1}$} if $f$ is of class $C^1$ and $f'$ is Lipschitz continuous. Let $U \subseteq \R^n$ be open. A map $f\colon U \to \R$ is a \begriff{$C^{1,1}$-submersion on its zero set} if, for every zero $x$ of $f$, $f$ is 
of class $C^{1,1}$ on a neighborhood of $x$ and $f'(x)$ is surjective \cite[\citeDefinition~4.52]{Zeidler.i}. By a {sequence} \begriff{$(x_k)_{k \in \mathbb{N}}$} in $X$ we mean a map $x\colon\mathbb{N}\to X$.\looseness=-1

\section{Convexity and strong convexity of sublevel sets} 
\label{s:mainresultssublevelsets}
Our main results are presented in Section \ref{ss:mainresults} together with a brief discussion of their hypotheses, and their application is illustrated in Section \ref{ss:applications}. The proofs of our main results, however, are postponed to Section \ref{ss:proofs} as 
several auxiliary results will be needed for both proofs.

We will assume the following for the sublevel sets $\Omega$ that we consider.
\begin{hypothesis}
\label{hy:g}
\newcommand{\raum}{\R^n}
$U \subseteq \raum$ is open, $m \in \mathbb{N}$, the maps
$g_1,\ldots,g_m \colon U \to \mathbb{R}$ are 
$C^{1,1}$-submersions on their zero sets as well as continuous, and the set $\Omega$ defined by \ref{e:sublevelset} is closed in $\R^n$ and connected.
\end{hypothesis}
We remark that the assumption on connectedness is essential for the main results below to be true. Although proving connectedness might be difficult in general, the sublevel sets considered in our applications (Sections \ref{ss:applications} and \ref{s:application}) will be connected by their definition.

For the maps $g_1,\ldots,g_m$ defining $\Omega$ in \ref{e:sublevelset} and for $x \in \Omega$ we denote  
\begin{align*}
\mathcal{A}(x)&=\{ i \in \{1,\ldots,m\} \ | \ g_i(x) =0\}, \\ 
\mathcal{C}(x)&=
\{ h \in \R^n \ | \ \forall_{j \in \mathcal{A}(x)} \ g'_j(x) h \leq 0\}.
\end{align*}
\subsection{The main results} We begin with the main result for ordinary convexity. 
\label{ss:mainresults}
\begin{theorem}
\label{thm:sublevelset:sufficiency:convex}
Let $\Omega$ be defined by \ref{e:sublevelset}. Assume \reftohypothesis{\ref{hy:g}}.
\begin{asparaenum}[(i)]
\item $\Omega$ is convex if for every $x \in \boundary \Omega$ we have $\interior \mathcal{C}(x) \neq \emptyset$ and there exists $i \in \mathcal{A}(x)$ such that 
\begin{equation}
\label{e:thm:sublevelset:sufficiency:convex}
\liminf_{t \to 0, t>0} \frac{g_i'(x+th)h}{t} \geq 0
\end{equation}
whenever $h \in \mathcal{C}(x) \cap \ker g_i'(x)$. 
\label{thm:sublevelset:sufficiency:i:convex}
\item 
\label{thm:sublevelset:sufficiency:iii:convex}
If $\Omega$ is convex, $x \in \boundary \Omega$, and the derivatives $g_i'(x)$ for every $i \in \mathcal{A}(x)$ are linearly independent, then \ref{e:thm:sublevelset:sufficiency:convex} holds for every $i \in \mathcal{A}(x)$ and all $h \in \mathcal{C}(x) \cap \ker g_i'(x)$. 
\end{asparaenum}
\end{theorem}
The analogous theorem for strong convexity is as follows.
\begin{theorem}
\label{thm:sublevelset:sufficiency}
Let $r>0$ and let $\Omega$ be defined by \ref{e:sublevelset}. Assume \reftohypothesis{\ref{hy:g}}.
\begin{asparaenum}[(i)]
\item
\label{thm:sublevelset:sufficiency:i}
$\Omega$ is $r$-convex if for every $x \in \partial \Omega$ we have
$\interior \mathcal{C}(x) \neq \emptyset$ and there exists
$i \in \mathcal{A}(x)$ such that
\begin{equation}
\label{e:thm:sublevelset:sufficiency}
\liminf_{t\to 0,t>0}\frac{g'_i(x+th)h}{t} \geq \frac{1}{r}\|g_i'(x)\|\cdot \|h\|^2
\end{equation}
whenever $h \in \mathcal{C}(x)\cap \ker g'_i(x)$.
\item
\label{thm:sublevelset:sufficiency:ii}
$\Omega$ is $r$-convex if \ref{e:thm:sublevelset:sufficiency} holds
for every $x \in \boundary \Omega$, every $i \in \mathcal{A}(x)$, and
all $h \in \mathcal{C}(x)\cap \ker g'_i(x)$.
\item
\label{thm:sublevelset:sufficiency:iii}
If $\Omega$ is $r$-convex, $x \in \partial \Omega$, and the
derivatives $g_i'(x)$ for $i \in \mathcal{A}(x)$ are linearly
independent, then \ref{e:thm:sublevelset:sufficiency} holds for every
$i \in \mathcal{A}(x)$ and all
$h \in \mathcal{C}(x)\cap \ker g'_i(x)$.
\end{asparaenum}
\end{theorem}
\myvspace
An immediate consequence of Theorem \ref{thm:sublevelset:sufficiency} is the next corollary which will be important for the applications in the next subsection and in Section \ref{s:application}. 
\begin{corollary}
\label{thm:sublevelset:g}
Let $m=1$, $r>0$, and let $\Omega$ be defined by \ref{e:sublevelset}. Assume \reftohypothesis{\ref{hy:g}}. Then $\Omega$ is $r$-convex if and only if 
\ref{e:thm:sublevelset:sufficiency}
holds for $i=1$, all $x \in \partial \Omega$ and all $h \in \ker g_1'(x)$.
\end{corollary}

\myvspace
We continue with several remarks on the above results. \\
If the map $g_i$ that is involved in \ref{e:thm:sublevelset:sufficiency:convex},\ref{e:thm:sublevelset:sufficiency} is of class $C^2$ rather than merely $C^{1,1}$ then the left hand sides of the latter inequalities reduce to $g_i''(x)h^2$. The assumption on the non-emptiness of the interior of the cones $\mathcal{C}(\cdot)$ in Theorem \ref{thm:sublevelset:sufficiency:convex}\ref{thm:sublevelset:sufficiency:i:convex} (in Theorem \ref{thm:sublevelset:sufficiency}\ref{thm:sublevelset:sufficiency:i}, respectively) and the linear independence of the derivatives in Theorem \ref{thm:sublevelset:sufficiency:convex}\ref{thm:sublevelset:sufficiency:iii:convex} (in Theorem \ref{thm:sublevelset:sufficiency}\ref{thm:sublevelset:sufficiency:iii}, respectively) are known as \begriff{constraint qualifications} in the field of optimization theory. Both assumptions are essential for the correctness of the statements as we will see in Examples \ref{ex:thm:convex:3} and \ref{ex:thm:convex:1}, respectively.
Assertions \ref{thm:sublevelset:sufficiency:i:convex} and \ref{thm:sublevelset:sufficiency:iii:convex} of Theorem \ref{thm:sublevelset:sufficiency:convex} correspond to assertions \ref{thm:sublevelset:sufficiency:i} and \ref{thm:sublevelset:sufficiency:iii} of Theorem \ref{thm:sublevelset:sufficiency} in the limit $r \to \infty$. However, the analogue to Theorem \ref{thm:sublevelset:sufficiency}\ref{thm:sublevelset:sufficiency:ii} does not hold in the setting of ordinary convexity in Theorem \ref{thm:sublevelset:sufficiency:convex}. (See Example \ref{ex:thm:convex:2}.)
We also note that both directions of \cite[\citeTheorem~1]{HenrionLouembet12} are untrue, in general, as Examples \ref{ex:thm:convex:1} and \ref{ex:thm:convex:2} reveal.

\vspace*{3pt}
\begin{example}
\label{ex:thm:convex:3}
Let $m = 2$, $U = \R^2$, $\Omega$ defined by \ref{e:sublevelset} and the maps $g_1,g_2 \colon \R^2 \to \R$ which are given by $g_1(x) = \|x\|^2 - 1$, $g_2(x)=-g_1(x)$. Then $\Omega = \boundary \cBall(0,1)$ is obviously not convex but \ref{e:thm:sublevelset:sufficiency:convex} is satisfied for $i=1$. The constraint qualification in the hypothesis of Theorem \ref{thm:sublevelset:sufficiency:convex}\ref{thm:sublevelset:sufficiency:i:convex} is not fulfilled since $\interior \mathcal{C}(x) = \emptyset$ for every $x \in \Omega$. Indeed, $\mathcal{C}(x) = \ker g'_1(x) = \ker g'_2(x)$. 
\end{example}

\begin{example}
\label{ex:thm:convex:1}
Let $g_1,g_2 \colon \R^2 \to \R$ be given by $g_1(x_1,x_2) = x_1^2+x_2^2 -1$ and $g_2(x_1,x_2) = - ( (x_1-2)^2 + x_2^2 -1)$. Let $\Omega$ be given by \ref{e:sublevelset} with $m = 2$ and $U = \R^2$. It is easy to see that $\Omega = \cBall(0,1)$, hence $\Omega$ is convex. For $h = (0,1)^\ast$ we have $g_2 ''(1,0)h^2 = -2 < 0$. However, in \cite[\citeTheorem~1]{HenrionLouembet12} the converse inequality was claimed. The assumption on the linear independence in Theorem \ref{thm:sublevelset:sufficiency:convex}\ref{thm:sublevelset:sufficiency:iii:convex} is not satisfied at the point $(1,0)$, and condition \ref{e:thm:sublevelset:sufficiency:convex} does not hold for $i=2$.
\end{example}

Examples \ref{ex:thm:convex:3} and \ref{ex:thm:convex:1} illustrate similarly that the constraint qualifications in Theorem \ref{thm:sublevelset:sufficiency} cannot be dropped. The next example shows that the analogous statement to Theorem \ref{thm:sublevelset:sufficiency}\ref{thm:sublevelset:sufficiency:ii} does not hold for ordinary convexity.
\begin{example}
\label{ex:thm:convex:2}
Let $g_1,\ldots,g_5 \colon \R^3 \to \R$ be given by 
\begin{align*}
g_{1}(x_1,x_2,x_3) &= x_1^3x_2^3 + x_3, \ g_2 = - g_1, \\
g_3(x_1,x_2,x_3) &= x_1, \\
g_{4}(x_1,x_2,x_3) &= x_3, \ g_5 = -g_4,
\end{align*}
and $\Omega$ as in \ref{e:sublevelset} with $m=5$, $U=\R^3$. According to the above definitions, 
\begin{equation*}
\Omega = \{ (x_1,x_2,x_3) \in \R^3 \ | \ x_1 \leq 0,x_2 = x_3 = 0\} \cup  \{ (x_1,x_2,x_3) \in \R^3 \ | \ x_1 = x_3 = 0\},
\end{equation*}
which is obviously not convex, in contrast to what was claimed in \cite[\citeTheorem~1]{HenrionLouembet12}. Note that \ref{e:thm:sublevelset:sufficiency:convex} holds at $x \in \boundary\Omega$ for \textit{all} indices $i \in \{1,\ldots,5\}$ and all $h \in \ker g_i(x)$. The calculations are straightforward. 
\end{example}
\subsection{Immediate applications} 
\label{ss:applications}
As we have already discussed, the need for convexity criteria is ubiquitous, e.g. \cite{Lasserre10,HenrionLouembet12,AhmadiOlshevskyParriloTsitsiklis13}. In this paper we restrict ourselves to applications of our main result on strong convexity. The case of attainable sets of nonlinear differential equations, to be discussed in Section \ref{s:application}, will be the main application. We emphasize, however, that our main results consitute
powerful tools for a much broader field of applications. This
is demonstrated here by recovering, quite conveniently, two
results from the literature.

\begin{example}
Let $P \in \mathbb{R}^{n \times n}$ be a symmetric and positive
definite matrix. Let $\mu_+(P)$ and $\mu_{-}(P)$ denote the maximum and minimum, respectively, eigenvalues of $P$, and let the ellipsoid $\Omega$ be given by
$
\Omega = \{ x \in \mathbb{R}^n | \innerProd{x}{Px} \leq 1 \}.
$
Then the condition
\begin{equation}
\label{e:ex:radius}
r
\geq
\mu_+(P)^{1/2} / \mu_-(P)
\end{equation}
implies that $\Omega$ is $r$-convex \cite[\citeTheorem~3]{Polovinkin96}. The condition is, in fact, both sufficient and necessary.
To see this, define $g(x) = \innerProd{x}{Px}-1$ to obtain $g'(x) h = 2 \innerProd{h}{Px}$ and $g''(x)h^2 = 2 \innerProd{h}{Ph}$. Then $g(x)=0$ implies $\| g'(x) \| \leq 2 \| P^{1/2}\|$. Hence, we obtain the bound \looseness=-1
\begin{equation}
\label{e:ex:P}
\frac{\|g'(x)\| \cdot \|h\|^2}{g''(x)h^2} \leq \frac{\mu_+(P)^{1/2}}{\mu_-(P)}
\end{equation}
for $h \neq 0$, which is attained if $x$ and $h$ are eigenvectors
corresponding to eigenvalues $\mu_+(P)$ and $\mu_{-}(P)$, respectively. By Corollary \ref{thm:sublevelset:g}, $\Omega$ is $r$-convex if and only if \ref{e:ex:radius} holds.
\end{example}

As already mentioned in the introduction, our result also recovers \cite[\citeTheorem~12]{JourneeNesterovRichtarikSepulchre08}. To show this, we need the following definition \cite{Plis74,JourneeNesterovRichtarikSepulchre08}.

\begin{definition}
\label{def:plis}
\newcommand{\PlisParameter}{\sigma}
\newcommand{\raum}{\R^n}
\newcommand{\orthogonal}[2]{#1\perp #2}
Let $\PlisParameter>0$. A set $\Omega \subseteq \raum$ is called \begriff{$\PlisParameter$-regular} if 
\begin{equation}
\label{e:epsregular}
\cBall(\alpha x + (1-\alpha)y,\PlisParameter \! \cdot \! \alpha (1-\alpha) \|x-y\|^2)\subseteq \Omega
\end{equation}
for all $x,y\in \Omega$ and all $\alpha \in \intoo{0,1}$. 
\end{definition}

The following result has been given in \cite{Vial83} without proof. Note that $r$ and $\sigma$ are in reciprocal proportion in the statement.
{
\newcommand{\raum}{\R^n}
\begin{lemma}
\label{l:equivalence:rConv:epsregular}
Let $\Omega \subseteq \raum$ be closed, and $r>0$. Then $\Omega$ is $r$-convex if and only if $\Omega$ is $\sigma$-regular for $\sigma = 1/(2r)$.
\end{lemma} 
}
\begin{proof}
\newcommand{\PlisParameter}{\sigma}
\newcommand{\raum}{\R^n}
\newcommand{\orthogonal}[2]{#1\perp #2}
For every $x \in \Omega$ the quantity $\delta_{\Omega,x}^\circ(\varepsilon)$ defined by the formula
\begin{align*}
\sup \big \{ \delta \geq 0 \; \big | \; y \in \Omega, \
&v \in \raum, \ \|v\| =1, 
\ \|x-y\|=\varepsilon, \\ &\orthogonal{v}{(x-y)}  \Longrightarrow {(x+y)}/{2} +\delta \cdot v \in \Omega \big \}
\end{align*} 
induces a map $\delta_{\Omega,x}^\circ$ on the non-negative real numbers, where
we have adopted the convention $\sup \emptyset =- \infty$.
Then
$
\delta^\circ_{\Omega,x}(\varepsilon) \geq \varepsilon^2 / (8r)
$
for all $\varepsilon>0$ if $\Omega$ is $1/(2r)$-regular,
so $\liminf_{\varepsilon \to 0,\varepsilon>0} \delta^\circ_{\Omega,x}(\varepsilon)/\varepsilon^2 \geq 1/(8r)$.
Since $\Omega$ is convex it follows that $\Omega$ is $r$-convex \cite[\citeTheorem~2.2]{i12sConvex}. \\
Conversely, assume that $\Omega$ is $r$-convex, and $\Omega \neq \R^n$ without loss of generality. Let $x,y \in \Omega$, hence $\|x-y\|\leq 2r$. In view of \ref{e:epsregular} and \ref{linse}, we need to show $\cBall(z,\rho) \subseteq \cBall(c,r)$ for $z = \alpha x + ( 1 - \alpha ) y$, $\rho = (2r)^{-1} \alpha (1-\alpha) \varepsilon^2$, $\varepsilon = \| x-y\|$, $\alpha \in \intoo{0,1}$ and $c \in \R^n$ such that $x,y \in \cBall(c,r)$. It is easy to see that it is enough to consider $c \in \R^n$ such that $x,y \in \boundary \cBall(c,r)$. Let $x+y = 0$ without loss of generality. To show $\cBall(z,\rho) \subseteq \cBall(c,r)$, it is enough to prove $\rho + \|c-z\|\leq r$. Suppose the reverse inequality for some $\alpha$. It is straightforward to conclude 
\begin{align*}
\|c-z\|^2 = r^2 - \varepsilon^2/4 + \|z\|^2 = r^2 - \varepsilon^2 \alpha ( 1 - \alpha),
\end{align*}
hence $(2r)^{-1} \alpha ( 1 - \alpha) \varepsilon^2 + r \sqrt{1 - \varepsilon^2 \alpha ( 1- \alpha)/r^2}>r$ by the assumption on $\alpha$. That is, $\sqrt{1 - \varepsilon^2 \alpha (1-\alpha)/r^2} > 1 - (2r^2)^{-1} \alpha ( 1- \alpha)\varepsilon^2$, which is a contradiction. 
\end{proof}

\myvspace
We now recover \cite[\citeTheorem~12]{JourneeNesterovRichtarikSepulchre08} in the example below.
\begin{example}
\newcommand{\raum}{\R^n}
Let $f \colon \raum \to \R_+$ be of class $C^{1,1}$, let $L>0$ be a
Lipschitz constant for $f'$, and let $\sigma>0$ be such that for any $x, h \in \raum$ it holds
$
f(x+h) \geq f(x) + f'(x)h + \frac{\sigma}{2}\|h\|^2.
$
Then the set 
$
\Omega_\omega :=\{ x \in \R^n \ | \ f(x) \leq \omega \}
$
is $r$-convex if 
\begin{equation}
\label{e:Journee:radius}
r \geq  \sigma^{-1} \sqrt{2L \omega}
\end{equation}
and $\omega > 0$. This follows from \cite[\citeTheorem~12]{JourneeNesterovRichtarikSepulchre08} and Lemma \ref{l:equivalence:rConv:epsregular}.
In order to prove the result using Corollary \ref{thm:sublevelset:g},
assume without loss of generality that $\Omega_\omega$ consists of
more than one point. Then $f(x) = \omega$ implies $f'(x) \neq 0$ since
otherwise $f(x+h) \geq \omega + (\sigma/2) \|h\|^2$ for all $h \in
\R^n$, which is a contradiction as we excluded the case of
$\Omega_\omega$ being a singleton. Moreover, $\|f'(x)\| \leq \sqrt{2 L
  f(x)}$ \cite[\citeProposition~11(ii)]{JourneeNesterovRichtarikSepulchre08} and $\liminf_{t \to 0, t>0}
f'(x+th)h/t \geq \sigma \|h\|^2$ \cite[\citeChapter~IV]{HiriartUrrutyLemarechal93i}. Using \ref{e:thm:sublevelset:sufficiency} the proof is finished.\\
One can also show that for any particular $\omega > 0$, Corollary
\ref{thm:sublevelset:g} provides a bound on $r$ that is better than the bound \ref{e:Journee:radius}, which has been obtained in \cite{JourneeNesterovRichtarikSepulchre08}. 
\end{example}
\subsection{Proof of the main results}
\label{ss:proofs}
In the present subsection we prove Theorems \ref{thm:sublevelset:sufficiency:convex} and \ref{thm:sublevelset:sufficiency}. We need several auxiliary results, to be presented in Sections \ref{sss:qsupp}--\ref{sss:topology}, and the proofs of the main results are completed in Section \ref{sss:proofs}. 
\subsubsection{The quadratic support property}
\label{sss:qsupp}
In the proof of Theorem \ref{thm:sublevelset:sufficiency} we will make use of the \begriff{local quadratic support}, which we introduce below. (See also \ref{fig:QuadSupport}.) It is a local characterization of strong convexity. We will need the following theorem of \person{Mayer} \cite{Mayer35}, which is also a local characterization.
\begin{theorem}
\label{thm:mayer}
An open and connected set $\Omega\subseteq \R^n$ is $r$-convex if and only if it is spherically supported at each of its boundary points locally, by which we mean that for any $x \in \boundary \Omega$ there exists a neighborhood $U\subseteq \R^n$ of $x$ and some $v \in \R^n$, $\|v\| = 1$, such that $U \cap \Omega \subseteq \cBall(x-rv,r)$.
\end{theorem}

See also \cite[\citeTheorem~1.2]{i12sConvex} for a proof of Theorem \ref{thm:mayer}.\\
The idea of local quadratic support is to replace a supporting ball at $x \in \boundary \Omega$ by its second-order parabolic approximation at $x$. 
\begin{definition}
\label{def:LocalQuadSupport}
Let $r > 0$ and $\Omega \subseteq \mathbb{R}^n$.
The vector $v \in \mathbb{R}^n$ \begriff{quadratically supports $\Omega$ with radius $r$ at $x \in \boundary\Omega$ locally} if $\| v \| = 1$ and there exists a neighborhood
$U \subseteq \mathbb{R}^n$ of $x$ such that
\begin{equation}
\label{e:def:QuadSupport}
\| h \|^2 \leq 2 r \mu
\end{equation}
whenever $h \in \mathbb{R}^n$, $h \perp v$, $\mu \in \mathbb{R}$ and
$x + h - \mu v \in U \cap \Omega$.
\\$\Omega$ is \begriff{quadratically supported with radius $r$ at
$x \in \boundary \Omega$ locally} if the previous condition is
satisfied for some $v$.
\end{definition}

\begin{figure}
\begin{center}
  \psfrag{v}[r][r]{$v$}
  \psfrag{x}[b][b]{$x$}
  \psfrag{h}[t][t]{$h$}
  \psfrag{U}[][]{$U$}
  \psfrag{O}[t][t]{$\Omega$}
  \psfrag{q}[r][r]{$-\mu v$}
  \psfrag{p}[r][r]{$~$}
  \psfrag{z}[l][l]{$x+h-\frac{v}{2r}\|h\|^2$}
  \psfrag{O}[t][t]{$\Omega$}
  \psfrag{C}[][]{$U \cap \Omega$}
  \includegraphics[width=.6\linewidth]{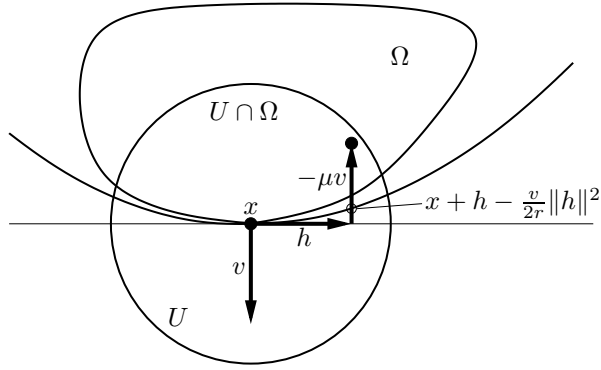}
\end{center}
\caption{\label{fig:QuadSupport}
Illustration of the property of local quadratic support \cite{i13qsuppc}. See
Def.~\ref{def:LocalQuadSupport}.}
\vspace*{-\baselineskip}%
\end{figure}
Although our notion is point-wise not equivalent to the spherical support property, it becomes equivalent when assumed at all boundary points:
\begin{theorem}
\label{th:rConvQuadSupp}
Let $r > 0$ and $\Omega \subseteq \mathbb{R}^n$ be open and
connected. Then the following conditions are equivalent:
\begin{enumerate}
\item
\label{th:rConvQuadSupp:rConv}
$\Omega$ is $r$-convex.
\item
\label{th:rConvQuadSupp:QuadSuppAnyNormal}
$\Omega$ is convex, and for all $x \in \boundary \Omega$, every unit
normal $v$ to $\Omega$ at $x$ quadratically supports $\Omega$ with
radius $r$ at $x$ locally.
\item
\label{th:rConvQuadSupp:QuadSupp}
$\Omega$ is quadratically supported with radius $r$ at each of its
boundary points locally.
\end{enumerate}
\end{theorem}
\begin{proof}
\newcommand{\orthogonal}[2]{\innerProd{#1}{#2}=0}
\newcommand{\raum}{\R^n}
We may assume $\Omega \neq \R^n$ without loss of generality.
Let $\Omega$ be $r$-convex and $x \in \partial \Omega$. 
Then $\Omega$ is convex, and $\Omega$ is spherically supported with radius $r$ at each of its boundary points by Proposition \ref{p:globalchar}. Let $v \in \R^n$ be a unit normal to $\Omega$ at $x$. Observe that $h \in \mathbb{R}^n$ 
such that $\orthogonal{h}{v}$ and $x+h-\mu v \in \Omega$ implies 
\begin{align*}
r^2 \geq \| x+h-\mu v - (x-rv)\|^2 = \|h\|^2+(\mu-r)^2\geq \|h\|^2 -2\mu r+r^2,
\end{align*}
hence \ref{e:def:QuadSupport} as claimed in \ref{th:rConvQuadSupp:QuadSuppAnyNormal}. Implication \ref{th:rConvQuadSupp:QuadSuppAnyNormal}$\Rightarrow$\ref{th:rConvQuadSupp:QuadSupp} is trivial, so it is left to prove \ref{th:rConvQuadSupp:QuadSupp}$\Rightarrow$\ref{th:rConvQuadSupp:rConv}. Let $s>r$ and suppose there exists a sequence $\sequence{x}{\Omega}{~}$ 
converging to $x$ such that $s^2 < \|x_k-(x-sv)\|^2$ and $x_k \neq x$ for all $k \in \mathbb{N}$. $\Omega$ is quadratically supported with radius $r$ at $x$ locally. Thus, there exists a sequence $(h,\mu)\colon \mathbb{N} \to \R^n \times \R$ and a unit normal $v \in \raum$ 
such that $\orthogonal{v}{h_k}$, $x_k =x+h_k -\mu_k v$ and $\|h_k\|^2 \leq 2r\mu_k$ for all $k$. 
We conclude $$s^2 < \|h_k  - \mu_k v + sv\|^2 \leq 2r\mu_k + (s-\mu_k)^2,$$
hence $0< 2(r-s) + \mu_k$ for any $k$, which is a contradiction as $\mu_k \to 0$. So, $\Omega$ is spherically supported with radius $s$ at $x$ locally. The proof is completed by Theorem \ref{thm:mayer} and the subsequent lemma, which is easily established. 
\end{proof}
\begin{lemma}
\label{lem:rconv:sequence}
Let $\Omega \subseteq \mathbb{R}^n$ be closed or open and let
$(r_i)_{i \in \mathbb{N}}$ be a sequence of reals converging to
$\rho > 0$.
If $\Omega$ is $r_i$-convex for all $i \in \mathbb{N}$, then $\Omega$
is $\rho$-convex.
\end{lemma}

\myvspace
A useful consequence of Theorem \ref{th:rConvQuadSupp} and Lemma \ref{lem:rconv:sequence} is the following.
\begin{corollary}
\label{cor:qsupp:clintOmega}
Let $r>0$ and $\Omega \subseteq \R^n$ be closed such that $\Omega = \closure( \interior \Omega )$ and $\interior \Omega$ is connected. Then $\Omega$ is $r$-convex if $\Omega$ is quadratically supported with radius $s$ at each of its boundary points locally for any $s>r$.
\end{corollary}

\myvspace
We proceed with relating the property of local quadratic support to optimality. For this purpose we will need a second-order sufficient condition in constrained optimization with $C^{1,1}$-data: Let us consider the optimization problem
\begin{equation}
\label{e:optimizationproblem}
\tag{OP}
\min_{x \in \Omega} f(x)
\end{equation}
and assume the following.
\begin{hypothesis}
\label{h:op}
$U \subseteq \R^n$ is open, and in \ref{e:optimizationproblem}, $f \colon U \to \R$ is of class $C^{1,1}$, $\Omega$ is of the form \ref{e:sublevelset} with $g_1,\ldots,g_m \colon U \to \R$ being of class $C^{1,1}$. 
\end{hypothesis}

As usually, we say that $x \in \R^n$ is a \begriff{local minimum point} 
\begriff{of} \ref{e:optimizationproblem} if $x \in \Omega$ and there exists a neighborhood $V \subseteq \R^n$ of $x$ such that $f(y)\geq f(x)$ 
whenever $y \in V\cap \Omega \setminus \{x \}$. 

\myvspace
A first order sufficient condition for optimality with constraints is as follows.
{
\newcommand{\xmin}{x_0}
\begin{theorem}
\label{t:sosc}
Assume \reftohypothesis{\ref{h:op}}. Let $(\xmin,\lambda) \in \Omega \times \R^m$ be such that for the function $L_\lambda \colon U \to \R$ given by $L_\lambda(x) = f(x) + \sum_{i = 1}^m \lambda_i g_i(x)$ it holds that $L_\lambda '(\xmin) = 0$, $\lambda_i \geq 0$ and $\lambda_i g_i(\xmin) = 0$ for all $i\in \{1,\ldots,m\}$. Then $\xmin$ is a local minimum point of \ref{e:optimizationproblem} if 
\begin{equation}
\label{e:t:sosc:condition}
\liminf_{t\to 0,t>0} \textfrac{L_\lambda '(\xmin +th)h}{t} > 0
\end{equation}
for all $h \in \mathcal{C}(\xmin) \cap \bigcap_{i : \lambda_i > 0} \ker g_i'(\xmin) \setminus \{0 \}$, where the intersection is formed over the indices $i\in \{1,\ldots,m\}$ that satisfy $\lambda_i > 0$.
\end{theorem}
\begin{proof}
We shall establish the inequality
\begin{equation}
\label{e:LaTorreRocca}
\liminf_{t\to 0,t>0} {L_\lambda'(\xmin + th)h}/{t} \leq \liminf_{t\to 0,t>0} {(L_\lambda(\xmin + th)-L_\lambda(\xmin))}/{(t^2/2)}
\end{equation}
for all $h \in \R^n$, which implies the condition (14) in \cite[\citeTheorem~3.2]{LiXu10}, and so the theorem follows from the latter sufficient optimality condition. We adopt a main argument of the proof of \cite[\citeTheorem~4]{LaTorreRocca03} to establish \ref{e:LaTorreRocca}. For $t>0$ sufficiently small, let $\varphi_1(t) = L_\lambda(\xmin + th)$ and $\varphi_2(t) = t^2$. By the extended mean value theorem, e.g. \cite{Walter04}, for each such $t$ there exists $\xi \in \intoo{0,t}$ satisfying 
\begin{equation}
\label{e:LaTorreRocca2}
2 \cdot \frac{\varphi_1(t)-\varphi_1(0)}{\varphi_2(t)-\varphi_2(0)}= 2 \cdot \frac{\varphi_1'(\xi)}{\varphi_2'(\xi)} = \frac{L_\lambda ' (\xmin + \xi h)h}{\xi}.
\end{equation}
The left hand side of \ref{e:LaTorreRocca2} equals $ {(L_\lambda(\xmin + th)-L_\lambda(\xmin))}/{(t^2/2)}$. The right hand side of \ref{e:LaTorreRocca2} is bounded from below by $\liminf_{\tau\to 0,\tau>0} L_\lambda '(\xmin + \tau h)h/\tau$ for any sufficiently small $t>0$. Thus, we established \ref{e:LaTorreRocca}. 
\end{proof}
}

The following lemma will be the key to the proof of Theorem \ref{thm:sublevelset:sufficiency}\ref{thm:sublevelset:sufficiency:i} and \ref{thm:sublevelset:sufficiency:ii}. (Roughly speaking, due to the strict inequality in \ref{e:t:sosc:condition} the proof of Theorem \ref{thm:sublevelset:sufficiency:convex}\ref{thm:sublevelset:sufficiency:i:convex} requires a different preliminary result. See Section \ref{sss:maxfunction}.) 
\begin{lemma} 
\label{l:optimalityandqsupp}
Let $s>0$ and $\Omega$ be defined as in \ref{e:sublevelset}, assume \reftohypothesis{\ref{hy:g}}, $0 \in \boundary \Omega$ and $1 \in \mathcal{A}(0)$. The vector $v \defas g_1'(0)^\ast / \|g_1 ' (0)\|$ quadratically supports $\Omega$ with radius $s$ at $0$ locally in each of the following cases:
\begin{asparaenum}[(i)]
\item 
\label{l:optimalityandqsupp:ii}
$0$ is a local minimum point of \ref{e:optimizationproblem} with $f \colon U \to \R$ given by 
\begin{equation*} 
f(x) = -\|g_1'(0)\| \cdot \big (\innerProd{v}{x} + \| x - \innerProd{v}{x}\!\cdot\!v \|^2 /(2s) \big ).
\end{equation*}
\item 
\label{l:optimalityandqsupp:iii}
The inequality
\begin{equation*}
\liminf_{t\to 0,t>0}\frac{g'_1(th)h}{t} > \frac{1}{s}\|g_1'(0)\|\cdot \|h\|^2
\end{equation*}
holds whenever $h \in \mathcal{C}(0)\cap \ker g'_1(0) \setminus \{0\}$.
\end{asparaenum}
\end{lemma}

\myvspace
\begin{proof}
Let us show that \ref{l:optimalityandqsupp:ii} is sufficient by proving the contrapositive. Suppose $v$ does not quadratically support $\Omega$ at $0$ locally. Then there exists a sequence $(h,\mu)\colon \mathbb{N} \to \ker g_1'(0) \times \R$ such that $h_k$ converges to $0$, $h_k - \mu_k v \in \Omega$ and $\|h_k\|^2 > 2 s \mu_k$ for all $k$. It holds that
\begin{align*}
f( h_k - \mu_k v) = \|g_1'(0)\| \cdot (\mu_k - \|h_k\|^2/(2s)),
\end{align*} 
so $f(0)>f( h_k - \mu_k v)$, hence $0$ is not a local minimum point of \ref{e:optimizationproblem}. 
\\
To show that \ref{l:optimalityandqsupp:iii} is sufficient, we show that \ref{l:optimalityandqsupp:iii} implies \ref{l:optimalityandqsupp:ii}. For this purpose, we intend to use Theorem \ref{t:sosc}. Let $\lambda = (1,0,\ldots,0) \in \R^m$. The map $L_\lambda \colon U \to \R$ given by $L_\lambda(x) = f(x) + g_1(x)$ satisfies $L_\lambda(0) = 0$ and $L_\lambda'(0)= 0$ since for $y \in U$, $h \in \R^n$ we have
\begin{equation*}
\label{}
f'(y)h = -\|g_1'(0)\| \cdot \big (\innerProd{v}{h} +(y - \innerProd{v}{y} v)^\ast ( h -\innerProd{v}{h} v )/s \big ).
\end{equation*}
If we verify \ref{e:t:sosc:condition} for $h \in \mathcal{C}(0) \cap \ker g_1'(0) \setminus \{0 \}$, the proof is completed.
First note that $f$ is twice continuously differentiable, and so
\newcommand{\liminftopluszero}{\liminf_{t\to 0,t>0}}
\begin{align*}
\liminftopluszero \frac{L_\lambda'(th)h}{t} 
&\geq\liminftopluszero
\frac{f'(th)h}{t} + \liminftopluszero \frac{g_1'(th)h}{t}>0, 
\end{align*} 
where for the last inequality we used 
$\liminftopluszero \textfrac{f'(th)h}{t} = f''(0)h^2 =\linebreak -\frac{1}{s}\cdot \|g_1'(0)\|\cdot \|h-\innerProd{v}{h}\!\cdot \! v \|^2$,
and $\innerProd{v}{h} = 0$.
\end{proof}
\subsubsection{A lower bound for second-order derivatives of max-functions}
\label{sss:maxfunction}
Maximum functions appear quite naturally when considering sublevel sets of the form \ref{e:sublevelset}: An equivalent representation of \ref{e:sublevelset} is $\{x \in U\ | \ \max_{i \in \{1,\ldots,m\}} g_i(x) \leq 0 \}$. This is the reason for using maximum functions as a tool in the proof of Theorem \ref{thm:sublevelset:sufficiency:convex}\ref{thm:sublevelset:sufficiency:i:convex}. In this context, we need the next proposition. Before stating the result, we introduce generalized second-order directional derivatives.

Let $U \subseteq \R^n$ be open, and $f\colon U \to \R$ a function. We denote by $f'(x;h)$ the limit
$
\lim_{t \to 0,t>0} (f(x+th)-f(x))/t
$
if it exists for $x \in U$, $h \in \R^n$. Note that $f'(x;h) = f'(x)h$ if $f$ is differentiable. Furthermore, we set
\begin{align*}
\label{}
\overline{D}^2 f(x,h^2) &= \limsup_{t \to 0,t>0}\textfrac{(f'(x+th;h)-f'(x;h))}{t},\\
\underline{D}^2 f(x,h^2) &= \liminf_{t \to 0,t>0}\textfrac{(f'(x+th;h)-f'(x;h))}{t}.
\end{align*} 
\begin{proposition}
\newcommand{\activeset}{I}
\label{th:secondderivatives}
Let $m \in \mathbb{N}$, and $U \subseteq \R^n$ open and convex. Let $f_i \colon U \to \R$ be of class $C^{1,1}$, and $f\colon U \to \R$ be given by $f(x) = \max_{i \in \{1,\ldots,m\}} f_i(x)$. For $x \in U$ and $h \in \R^n$ define $\activeset(x) = \{ i \in \{1,\ldots,m\} \ | \ f_i(x) = f(x) \}$ and $M = \operatorname{argmax}_{i \in \activeset(x)} f'_i(x;h)$. Then the following inequality holds for all $x \in U$ and $h \in \R^n$:
\begin{equation}
\label{e:th:secondderivatives}
\overline D^2f(x,h^2) \geq \max_{i \in M} \underline D^2 f_i(x,h^2).
\end{equation}
\end{proposition}
Proposition \ref{th:secondderivatives} is a correction of \cite[\citeTheorem~7.5]{BednarikPastor04}:
\begin{example}
Let $f,f_1,f_2 \colon \R \to \R$ be defined by $f_1(x)=x^2-x$, $f_2(x) = x$ and $f(x) = \max_{i=1,2}f_i(x)$. Proposition \ref{th:secondderivatives} implies $\overline D^2 f(0,1^2)\geq 0$. Proposition \ref{th:secondderivatives} does not hold if $M$ is replaced by $I(x)$ in \ref{e:th:secondderivatives}, in general. Indeed, direct calculation leads to $\overline{D}^2f(0,1^2) = 0 < \max_{i=1,2} f_i''(0) = 2$. However, in \cite[\citeTheorem~7.5]{BednarikPastor04}, where $\overline{D}^2(x,h^2)$ is denoted by $f'^{u}(x;h,h)$, the reverse inequality is claimed.  Hence, the cited theorem is untrue, in general. (It was not revised in \cite{BednarikPastor06}.) Nevertheless, we adopt some ideas from \cite{BednarikPastor04} to establish Proposition \ref{th:secondderivatives}.
\end{example}
\begin{proof}[Proof of Proposition \ref{th:secondderivatives}]
First note that $f'(x;h) = \max_{i \in {I}(x)} f'_i(x;h)$ \cite[\citeCorollary~I.3.2]{DemyanovRubinov95}, and in particular $f'(x;h) = f'_i(x;h)$ for all $i \in M$. Let $i \in M$ and choose a sequence $\sequence{t}{\R_{+}\!\setminus\! \{0\}}{~}$ converging to $0$ such that \cite[\citeLemma~7.6]{BednarikPastor04}
$f'(x+t_kh;h)\geq f'_i(x+t_kh;h).$
Then
$
\underline D^2f_i(x,h^2) \leq \lim_{k \to \infty} \textfrac{(f'_i(x+t_kh;h)-f'_i(x;h))}{t_k},
$
without loss of generality, and
\begin{equation}
\label{e:ffi}
\frac{f'(x+ t_k h;h)-f'(x;h)}{t_k} \geq \frac{f_i'(x+t_k h;h)-f_i'(x;h)}{t_k}.
\end{equation}
By taking limits on both sides of \ref{e:ffi} the proof is finished. 
\end{proof}
\subsubsection{Regular closedness of connected sublevel sets and connectedness of their interiors} 
\label{sss:topology}
We prove the following property of sublevel sets that we consider.
\label{s:proof}
\begin{lemma}
\label{l:closureofinterior}
Let $\Omega$ be defined by \ref{e:sublevelset}. Assume \reftohypothesis{\ref{hy:g}} and $\interior \mathcal{C}(x) \neq \emptyset$ for all $x \in \boundary \Omega$. Then $\Omega$ is regular closed, i.e. $\Omega =\closure(\interior \Omega)$, and $\interior \Omega$ is connected.
\end{lemma}

\myvspace
The above lemma can be seen as the first step in the proofs of our main results. It allows one to use \cite[\citeTheorem~4.9]{Valentine64} and Corollary \ref{cor:qsupp:clintOmega} in the proof of Theorems \ref{thm:sublevelset:sufficiency:convex}\ref{thm:sublevelset:sufficiency:i:convex} and \ref{thm:sublevelset:sufficiency}\ref{thm:sublevelset:sufficiency:i}, respectively. The subsequent technical lemma will be needed for the proof of Lemma \ref{l:closureofinterior}.
\begin{lemma}
\label{l:star-shaped}
Let $U \subseteq \R^n$ be open, $g \colon U \to \R$ be of class $C^1$, $\Omega = \{ x \in U \ | \ g(x) \leq 0 \}$, and assume $g(0)=0$ and $g'(0)p < 0$ for some $p \in \R^n$. Then there exists some $\varepsilon>0$ such that the following holds for all $q \in \intoc{0,\varepsilon} \cdot p$: \\
$\conv\{0,q\}\setminus \{0\} \subseteq \interior \Omega$, and there exists an open neighborhood $V \subseteq \R^n$ of $\conv\{0 , q \}$ for which $V \cap \interior \Omega$ is star-shaped with respect to $q$. 
\end{lemma}
\begin{proof}
First choose $\varepsilon>0$ such that $\| x\| \leq \varepsilon \| p \|$ implies both $x \in U$ and 
$$
3 \| g'(0) -g'(x) \| < - g'(0) p / \|p\|.
$$
Let $q \in \intoc{0, \varepsilon} \cdot p$ and choose $\delta \in \intoc{0,1} \cdot \|q\|$ such that \looseness=-1
\begin{equation*}
-g'(0) p / \|p\| < -\frac{3}{2} g'(0) (q-y) /\|q-y\|
\end{equation*}
whenever $\|y\| < \delta$. The mean value theorem shows
\begin{equation*}
| g(y + t(q-y)) - g(y) - t g'(0) (q-y) | \leq - \frac{t}{2} g'(0) (q-y)
\end{equation*}
for all $y \in \oBall(0,\delta)$ and all $t \in \intcc{0,1}$.
Thus $g(y + t(q-y) ) < 0$ for all $y
\in \oBall(0,\delta) \cap \Omega$ and all $t \in \intoc{0,1}$,
that is,
\begin{equation}
\label{l:star-shaped:proof:*}
\conv \{ y, q \} \setminus \{ y \} \subseteq \interior \Omega
\end{equation}
holds for all $y \in \Omega \cap \oBall(0,\delta)$. Now choose
$\gamma \in \intoo{0,\delta/2}$ such that
$\oBall(q,\gamma) \subseteq \Omega$ and
$\oBall(\delta q/2/\| q\|, \gamma) \subseteq \Omega$, and set
$V = \oBall(q,\gamma) \cup \conv (\{ q \} \cup \oBall(0,\gamma))$.
As $V$ is star-shaped with respect to $q$ it remains to show that
\ref{l:star-shaped:proof:*} holds whenever
$y \in V \cap \interior \Omega$.
That implication is obvious if additionally
$y \in \oBall(q,\gamma) \cup \oBall(0,\delta)$. Otherwise,
$y \in \conv ( \{ q \} \cup \oBall(0,\gamma) ) \cap \interior \Omega$
and $\| y \| \geq \delta$. Then
$\conv \{ y, q \} \subseteq \conv \{ z, q \}$ for some
$z \in \oBall(\delta q /2 /\| q \|,\gamma)
\subseteq \Omega \cap \oBall(0,\delta)$, so
$\conv \{ z, q \} \subseteq \interior \Omega$ by
\ref{l:star-shaped:proof:*}, which completes the proof.
\end{proof}

\begin{proof}[Proof of Lemma \ref{l:closureofinterior}]
\newcommand{\raum}{\R^n}
The claim is trivial for $\Omega = \raum$, $\Omega = \emptyset$ and $\Omega$ a singleton, so we do not further consider these cases. Let $x \in \partial \Omega$ and without loss of generality, $x = 0$, $\mathcal{A}(0)=\{1,\ldots,m\}$. Let $p \in \interior \mathcal{C}(0)$. By assumption $g_i'(0)p<0$ for all $i \in \{1,\ldots , m\}$, and Lemma \ref{l:star-shaped} shows there exist $q \in \conv\{0,p\}\setminus \{0 \}$ and a neighborhood $V \subseteq \raum$ of $\conv\{0,q\}$ such that $V \cap \interior \Omega$ is star-shaped with respect to $q$, and $\conv\{0,q\} \setminus \{0\} \subseteq \interior \Omega$. The latter fact implies $\Omega = \closure (\interior \Omega)$, and the former, that for each $x_0 \in \boundary \Omega$ there exists a neighborhood $V \subseteq \raum$ of $x_0$ for which $V \cap \interior \Omega$ is connected. \\
In order to show $\interior \Omega$ is connected, let $O_1$ be a connected component of $\interior \Omega$ and assume $O_2$ = $\interior \Omega\setminus O_1$ is not empty. Then $\boundary O_1\cap \boundary O_2 \neq \emptyset$ since $\Omega = \closure O_1\cup \closure O_2$ and $\Omega$ is connected. So pick $x_0\in \boundary O_1\cap \boundary O_2$. Then there exists a sequence $\sequence{y}{O_2}{}$ that converges to $x_0$. But by the above argument, there exists a neighborhood $V\subseteq \R^n$ of $x_0$ for which $V\cap \interior \Omega$ is connected. This implies $y_k \in O_1$ for $k$ sufficiently large, which is a contradiction.
\end{proof}
\subsubsection{Proofs of Theorems \ref{thm:sublevelset:sufficiency:convex} and \ref{thm:sublevelset:sufficiency}}
\label{sss:proofs}
Now, we are in a position to prove our main results.
\begin{proof}[Proof of Theorem \ref{thm:sublevelset:sufficiency:convex}]
The claims are trivial for $\Omega = \R^n$, $\Omega = \emptyset$ and $\Omega$ a singleton, so we do not further consider these cases. 
Let us prove \ref{thm:sublevelset:sufficiency:i:convex}.\\
By Lemma \ref{l:closureofinterior}, we have $\Omega = \closure ( \interior \Omega)$. Therefore, if we assume $\Omega$ being non-convex then there exists $x \in \boundary\Omega$ that is not a point of mild convexity \cite[\citeTheorem~4.9]{Valentine64}. By the definition, this means that there exists $\zeta \in \R^n \setminus \{0\}$ such that 
\begin{equation*}
t \in [-1,1] \setminus \{0\} \ \Rightarrow \ x + t \zeta \in \interior \Omega.
\end{equation*}
Without loss of generality, let $x = 0$ and  $\mathcal{A}(0) = \{1,\ldots,m\}$. Choose some $v \in \R^n$ such that $g_i'(0)v<0$ for all $i\in \mathcal{A}(0)$. This choice is possible since $\interior \mathcal{C}(0) \neq \emptyset$. An application of the implicit function theorem provides functions $\mu_i \colon W_i \to \R$ on convex subsets $W_i \subseteq \ker g_i'(0)$ which represent the boundary of $\{x \in U \ | \ g_i(x) \leq 0 \}$ locally at $0$ in the sense that $g_i(h + v\mu_i(h)) = 0$ for all $h \in W_i$ and such that $g_i(h+\lambda v)\leq 0$ implies $\lambda \geq \mu_i(h)$ \cite{i07Convex}. Let $W = \cap_{i \in \mathcal{A}(0)} W_i$. Note that $\zeta \in W$, without loss of generality. Indeed, we have $\pm g_i'(0)\zeta = \lim_{t \to 0,t>0} g_i(\pm t\zeta)/t \leq 0$. \\
Define the function $\mu \colon W \to \R$ by $\mu(h) = \max_{i \in \mathcal{A}(0)}\mu_i(h)$, the point $h + v\mu_i(h)$ by $p_i(h)$, and the point $h + v\mu(h)$ by $p(h)$. Hence, $p(h) \in \partial \Omega$ for all $h\in W$. Note that $p(h),p_i(h)$ can be considered as maps $W \to \R^n$. Let $\nu \colon I \to \R$ be the function $t \mapsto \mu(t\zeta)$, where $I \defas \{ t \in \intcc{-1,1} | \, t\zeta \in W\}$. Note that $I=\intcc{-1,1}$ as $W$ is convex. The proof of \ref{thm:sublevelset:sufficiency:i:convex} is completed if we show that $\nu$ is convex on $I$. Indeed, together with the fact that $\zeta \in W \cap \interior \Omega$ this leads to the contradiction
\begin{equation*}
\label{}
0 = \nu(0) = \mu(-\textfrac{\zeta}{2} + \textfrac{\zeta}{2}) \leq \mu(-\zeta)/2 + \mu(\zeta)/2 < 0.
\end{equation*}
So let us show convexity of $\nu$. We remark that $D_{\!+}\nu\,(t) \defas \mu'(t\zeta;\zeta)$ exists\linebreak \cite[\citeCorollary~I.3.2]{DemyanovRubinov95}. $\nu$ is convex if $D_{\!+}\nu$ is increasing \cite[\citeTheorem~5.3.1]{HiriartUrrutyLemarechal93i}. For the latter, it is enough to prove that 
$
\overline D^2 \mu (t\zeta,\zeta^2)\geq 0
$
for all but countably many $t \in I$ \cite[\citeTheorem~12.24]{Walter04}. To verify the latter inequality, first observe that, by Proposition \ref{th:secondderivatives}, the inequality
\begin{equation}
\label{e:max}
\overline D^2 \mu(t\zeta,\zeta^2)\geq \max_{i \in \mathcal{A}(p(t\zeta))}\underline{D}^2\mu_{i}(t\zeta,\zeta^2)
\end{equation}
holds for $t \in I$ if $\mu_{i}'(t\zeta)\zeta = \mu_{j}'(t\zeta)\zeta$ for all $i,j \in \mathcal{A}(p(t\zeta))$ by Proposition \ref{th:secondderivatives}. However, for $i \neq j$ the set
$$
N_{i,j} = \{t \in I \ | \ \{i,j \} \subseteq \mathcal{A}(p(t\zeta)), \ \mu_{i}'(t\zeta)\zeta \neq \mu_{j}'(t\zeta)\zeta\}
$$
is discrete, and therefore countable. Indeed, let $\tau \in N_{i,j}$ be a limit point and consider the function $f(t) \defas \mu_{i}(t\zeta) - \mu_{j}(t\zeta)$. Thus, $f(\tau) = 0$ and $f'(\tau) \neq 0$. Any sequence $\sequence{\tau}{N_{i,j}}{~}$ converging to $\tau$ yields $\lim_{k \to \infty} f(\tau_k) / (\tau_k-\tau)= 0$, which is a contradiction. Hence, we have established \ref{e:max} for all but countably many $t \in I$. 
\\What is left to show is that the right hand side of \ref{e:max} is non-negative, at least for all but countably many $t \in I$. To this end, observe that the proof of \cite[\citeTheorem~3.1]{i07Convex} shows that $\underline{D}^2\mu_{i}(t\zeta,\zeta^2)$ equals 
\begin{equation}
\label{e:relationtogi}
-(g_{i}'(p_{i}(t\zeta))v)^{-1}\underline{D}^2g_{i}(p_{i}(t\zeta),(p_{i}'(t\zeta) \zeta)^2)
\end{equation}
for any $i$ \cite[\citeEquation~8]{i07Convex}. Thus, it remains to show that \ref{e:relationtogi} is non-negative for all $t \in I$ for which $t \notin \cup_{i \neq j} N_{i,j}$ holds. \\
By hypothesis and without loss of generality, let $i=1$ satisfy \ref{e:thm:sublevelset:sufficiency:convex} at the point $p(t\zeta)$ in place of $x$. We claim that
\begin{equation}
\label{e:vectorincone}
p'_{1}(t\zeta)\zeta \in \mathcal{C}(p(t\zeta))\cap \ker g_{1} '(p(t\zeta)).
\end{equation}
Indeed,
$p_1'(t\zeta)\zeta \in \ker g_j'(p(t\zeta))$ for any $j \in \mathcal{A}(p(t\zeta))$ since $p_1'(t\zeta)\zeta = p_j'(t\zeta)\zeta$ by the assumption on $t$, and therefore
\begin{equation}
\label{e:thm:sublevelset:sufficiency:convex:proof:zerorelation}
0 = g_j'(p(t\zeta))p_j'(t\zeta)\zeta = g_j'(p(t\zeta))p_1'(t\zeta)\zeta
\end{equation}
for any $j\in \mathcal{A}(p(t\zeta))$. The first equality in \ref{e:thm:sublevelset:sufficiency:convex:proof:zerorelation} is straightforward to establish. Now, inequality \ref{e:thm:sublevelset:sufficiency:convex} is true for $p(t\zeta)$ and $p'_{1}(t\zeta)\zeta$ at place of $x$ and $h$, respectively, by \ref{e:vectorincone}. Hence, the expression in \ref{e:relationtogi} is non-negative, and the proof of \ref{thm:sublevelset:sufficiency:i:convex} is finished.\myvspace \\
The proof of \ref{thm:sublevelset:sufficiency:iii:convex} we divide into 4 steps. Our purpose is to prove \ref{e:thm:sublevelset:sufficiency:convex} for $i = 1$ and $x = 0$, without loss of generality, and we will continue to use the map $\mu_1 \colon W_1 \to \R$ as defined above except with $v$ defined by $v=-g_1'(0)^\ast/\|g_1'(0)\|$. Let $f(h) = \mu_1(h)$. \\ \newcounter{stepcnt}
Step \refstepcounter{stepcnt}\label{step:1}\ref{step:1}. Suppose the claim was wrong, i.e. $\liminf_{t\to 0,t>0} f'(th)h/t < 0$ for some $h \in \ker g_1'(0) \cap \mathcal{C}(0) \setminus \{0\}$ as
$\liminf_{t\to 0,t>0} f'(th)h/t = \|g_1'(0)\|^{-1} \liminf_{t\to 0,t>0} g_1'(th)h/t$ \cite[\citeEquation~10]{i07Convex}. Then $\liminf_{t\to 0,t>0} f'(t\tilde h)\tilde h/t < 0$ whenever $\| \tilde h - h \|$ is small enough. Indeed, 
\begin{equation*}
f'(t\tilde h)\tilde h /t - f'(th)h/t = f'(th)(\tilde h - h)/t + (f'(t\tilde h)-f'(th))\tilde h /t,
\end{equation*}
hence
$
|f'(t\tilde h) \tilde h / t - f'(th)h/t|\leq  L \| h\|\! \cdot\! \| \tilde h - h \| + L \|\tilde h\|\! \cdot\! \|\tilde h - h\|,
$
where $L>0$ is a Lipschitz constant of $f'$ about the origin. \\
Step \refstepcounter{stepcnt}\label{step:2}\ref{step:2}. Since the derivatives $g'_i(0)$ are linearly independent there exists $\xi \in \R^n$ satisfying $g_1'(0)\xi = 0$ and $g_i'(0)\xi < 0$ for $i>1$. Now choose $\varepsilon>0$ small enough that
\begin{equation}
\label{eq:function:proof:2}
\liminf_{t\to 0,t>0} f'(t\tilde h)\tilde h/t < 0
\end{equation}
for $\tilde h := h + \varepsilon \xi$. \\
Step \refstepcounter{stepcnt}\label{step:3}\ref{step:3}. We define $z(t) = t \tilde h + \mu_1(t \tilde h)v$ for $t$ small enough. Then $z(t) \in \Omega$ for all $t\geq 0$ small enough. Indeed, $g_1(z(t)) = 0$ by the definition of $\mu_1$, and $g_i(z(t)) \leq 0$ for all $i>1$ since $g_i'(0) \tilde h <0$. \\
Step \refstepcounter{stepcnt}\label{step:4}\ref{step:4}. By Step \ref{step:3}, $z(t) \in \boundary \Omega$. As $\Omega$ is convex, we conclude furthermore that ${t \mapsto \mu_1(t \tilde h)}$ is convex for $t\geq 0$ small enough, which contradicts \ref{eq:function:proof:2}. 
\end{proof}
\begin{proof}[Proof of Theorem \ref{thm:sublevelset:sufficiency}]
The claims are trivial for $\Omega = \R^n$, $\Omega = \emptyset$ and $\Omega$ a singleton, so we do not further consider these cases. 
Let us prove \ref{thm:sublevelset:sufficiency:i}.\\
By Lemma \ref{l:closureofinterior} we can apply Corollary \ref{cor:qsupp:clintOmega}, so it suffices to prove that $\Omega$ is quadratically supported with radius $s$ at each of its boundary points locally for any $s>r$.
But this follows directly from Lemma \ref{l:optimalityandqsupp}.\\
For the proof of \ref{thm:sublevelset:sufficiency:ii}, let us use \ref{thm:sublevelset:sufficiency:i}. So we need to verify $\interior \mathcal{C}(0) \neq \emptyset$, without loss of generality. We set $v_i
= g_i'(0)^\ast/\|g_i'(0)\|$. Since \ref{e:thm:sublevelset:sufficiency} holds for every $i \in
\mathcal{A}(0)$ it follows as in the proof of \ref{thm:sublevelset:sufficiency:i} that all $v_i$ quadratically support $\Omega$ with radius $r$ at $0$ locally. Hence, $\oBall(0,\delta) \cap \Omega
\subseteq \mathcal{C}(0)$ for $\delta >0$ small enough. If $\interior \mathcal{C}(0) =
\emptyset$, then there exists some $i \in \{1, \ldots, m\}$ for which
${v_i}\perp{\mathcal{C}(0)}$ \cite[\citeTheorem~21.1]{Rockafellar70}. 
Local quadratic support of $\Omega$ at $0$ shows $0$ is an isolated point of $\Omega$, hence $\Omega = \{0\}$ as $\Omega$ is connected. This implies $\interior \mathcal{C}(0) \neq \emptyset$ as we had excluded the case of $\Omega$ being a singleton. \\
For the proof of \ref{thm:sublevelset:sufficiency:iii} we adopt Steps \ref{step:1} to \ref{step:3} of the proof of Theorem \ref{thm:sublevelset:sufficiency:convex}\ref{thm:sublevelset:sufficiency:ii} with $\tilde f$ at place of $f$ where $\tilde f(h) = \mu_1(h)- \|h\|^2/(2r)$. We modify Step \ref{step:4} as follows. \\
As we assumed in Step \ref{step:1} that the claim was wrong, there exists $s>r$ such that 
$\liminf_{t\to 0,t>0} \textfrac{\mu_1'(t\tilde h)\tilde h}{t} - \|\tilde h\|^2/s < 0$, 
hence
\begin{equation}
\label{e:thm:sublevelset:sufficiency:proof:laststep}
\liminf_{t\to 0,t>0} \frac{g_1'(t\tilde h)\tilde h}{t} < \frac{1}{s}\cdot \|g_1'(0)\| \cdot \|\tilde h\|^2
\end{equation}
We define a diffeomorphism $F\colon \R^n \to \R^n$ by $F(y) = y - v \cdot \frac{\innerProd{w}{y}^2}{r+s}$ with $w = \tilde h / \|\tilde h\|$ to obtain
$
F'(y)^{-1}F''(y) h^2 = 2v \frac{\langle{w}|{h}\rangle^2}{r+s},
$ 
hence $\|F'(y)^{-1} F''(y) h^2 \|\cdot (r+s)/2 \leq \|h\|^2$ for $h \in \R^n$. Then $F(\cBall(y,r))$ is convex for any $y\in \R^n$ \cite[\citeCorollary~1]{i07Convex}. This implies $F(\Omega)$ is convex since $\Omega$ is $r$-convex. We set $\hat f = g_1 \circ F^{-1}$ and remark that 
\begin{equation*}
F(z(t)) = t \tilde h + v \cdot ( \mu_1(t\tilde h) + t \|\tilde h\|/(r+s)).
\end{equation*}
Let $\hat \mu \colon W_1 \to \R$ be defined by $F(z(t)) = t \tilde h + v\hat \mu(t\tilde h)$. We have $0 = \hat f(t\tilde h + v \hat \mu (t \tilde h))$ for $t$ small enough, so $\liminf_{t\to 0,t>0} \hat \mu'(t\tilde h)\tilde h/t = \|\hat f'(0)\|^{-1} \liminf_{t\to 0,t>0} \hat f'(t\tilde h ) \tilde h /t$. $F(\Omega)$ is convex, $F(z(t)) \in \boundary F(\Omega)$ for $t\geq 0$, thus $t\mapsto \hat \mu(t\tilde h)$ is convex for $t\geq 0$ small enough. Hence, $\liminf_{t\to 0,t>0} \hat f'(t\tilde h)\tilde h / t \geq 0$. A simple calculation shows 
\begin{equation*}
\liminf_{t\to 0,t>0} \textfrac{g_1'(t\tilde h)\tilde h}{t} = \hat f'(0) F''(0) \tilde h^2 + \liminf_{t\to 0 , t>0} \textfrac{\hat f'(t\tilde h)\tilde h}{t}.
\end{equation*}
Then $\hat f'(0) F''(0) \tilde h^2 = 2 \|g_1'(0)\| \cdot \|\tilde h\|^2 /(r+s)$ and \ref{e:thm:sublevelset:sufficiency:proof:laststep} imply $\liminf_{t\to 0,t>0} \hat f'(t\tilde h)\tilde h/t<0$, which is contradiction.
\end{proof}
\section{Application to attainable sets of nonlinear systems}
\label{s:application}
In this section, we apply Corollary \ref{thm:sublevelset:g} to derive sufficient conditions for attainable sets of nonlinear systems of the form \ref{e:ode} to be strongly convex. The subsequent results practically apply to abstraction based controller design as already mentioned in the introduction. This 3-step procedure of controller design involves the computation of a so-called \begriff{discrete abstraction} as a first step before the subsequent steps of controller synthesis. A discrete abstraction is, roughly speaking, a finite state model which contains the behavior of the original system but, in general, much more spurious transitions. The method to compute discrete abstractions proposed in \cite{i11abs} requires the over-approximation of a large number of attainable sets. Whether the synthesis of a controller is successful or not, depends, among others, on the quality of the method used to over-approximate attainable sets.

In \cite{i11abs}, attainable sets were approximated by supporting half-spaces as the theory presented there gives sufficient conditions for attainable sets to be convex. Our results give sufficient conditions under which the attainable sets are strongly convex. As detailed in the introduction, attainable sets can then be approximated less conservatively by supporting balls instead of supporting half-spaces, where the balls can be obtained from the same data as the half-spaces. See \ref{fig:attset}.

We apply our results to abstraction based controller design in the example in the last part of this section. For simplicity, we state and prove in this paper the remaining theorems for the autonomous system \ref{e:ode}. Generalizations to the non-autonomous case are given in \cite{i13qsuppc}. 

We begin with a result on images of $C^{1,1}$-diffeomorphisms which follows from Corollary \ref{thm:sublevelset:g}.
{
\newcommand{\diffeo}{\Phi}
\begin{corollary}
\label{cor:diffeorad}
Let $\diffeo \colon U \to V$ be a $C^{1,1}$-diffeomorphism between open sets $U,V \subseteq \mathbb{R}^n$, $s>0$,
and $\Omega \subseteq U$ be $s$-convex and closed, $\Omega \neq \mathbb{R}^n$. Let $L_1,L_2 \in \mathbb{R}$, and assume that for 
each $x \in \partial \Omega$ there exists a unit normal $v$ to $\Omega$ at $x$ such that the 
following conditions hold for all $\xi \perp v$:
\begin{align}
L_1 \|\xi\|^2 &\geq \limsup_{t\to 0,t>0} \frac{\innerProd{v}{\diffeo'(x)^{-1}(\diffeo '(x+t\xi)\xi-\diffeo '(x)\xi)}}{t},
\label{eq:thm:diffeorad:1}
\\
L_2 \|\xi\|^2 & \geq \|\diffeo '(x) \xi \|^2 \cdot \|\diffeo '(x)^{-1}\|.\label{eq:thm:diffeorad:2a}
\end{align}
If $s L_1 < 1$ then $\diffeo (\Omega)$ is $r$-convex for 
$r = \textfrac{sL_2}{(1-sL_1)}$.
\end{corollary}
\begin{proof}
The claim is trivial for $\Omega = \emptyset$ and $\Omega$ a singleton, so we do not further consider these cases. \\ {
Let $x \in \boundary \Omega$ and let $v$ be as in the statement of the theorem. Define $g = f \circ \diffeo^{-1}$ with $f(z)=\|z-x + sv\|^2 - s^2$ . Then
\begin{equation}
\label{e:thm:diffeorad:3}
\diffeo(\Omega) \subseteq \{ y \in V \ | \ g(y) \leq 0 \}
\end{equation}
by Proposition \ref{p:globalchar}\ref{p:globalchar:ii}. $\Omega$ is closed, convex and $\interior \Omega \neq \emptyset$. This implies $\Omega = \closure ( \interior \Omega )$ and $\interior \Omega$ is connected. As $\Phi$ is a diffeomorphism we have $\Phi(\Omega ) = \closure ( \interior \Phi(\Omega))$ and $\interior \Phi(\Omega)$ is non-empty and connected. In view of Corollary \ref{cor:qsupp:clintOmega}, it suffices to prove that $\Phi(\Omega)$ is quadratically supported with radius $s$ at each of its boundary points locally. By \ref{e:thm:diffeorad:3} and the fact that $x$ is a boundary point of both $\Phi(\Omega)$ and\linebreak $\{ y \in V \ | \ g(y) \leq 0 \}$  it suffices to prove that the latter sublevel set is quadratically supported with radius $s$ at $x$ locally. To this end, assume $x = \Phi(x) = 0$ without loss of generality. By Corollary \ref{thm:sublevelset:g}, it is enough to prove \ref{e:thm:sublevelset:sufficiency} for all $h \in \ker g'(0)$. Now, differentiate the identity $f = g \circ \diffeo$, use the Lipschitz continuity of $g'$ and the continuity of $\diffeo '$ } to obtain that $2 \| \xi \|^2$ equals
\begin{equation*}
 \liminf_{t \to 0, t> 0 } \left ( \frac{g'(th)h}{t}+ 2 s \frac{\innerProd{v}{\diffeo ' (0)^{-1}(\diffeo '(t\xi)\xi-\diffeo ' (0)\xi)}}{t} \right )
\end{equation*}
whenever $h = \diffeo '(0) \xi$. Now $\|g'(0)\| = 2 s \|\diffeo ' (0)^{-1} \|$ and \ref{eq:thm:diffeorad:1},\ref{eq:thm:diffeorad:2a} imply \ref{e:thm:sublevelset:sufficiency} for all $h \in \ker g'(0)$. 
\end{proof}
}

We assume the following for the remaining theorems.
\begin{hypothesis}
\label{h:continuoustime}
Let $X \subseteq \mathbb{R}^n$ be an open set. The right hand side $F \colon X \to
\mathbb{R}^n$ of \ref{e:ode} is continuously
differentiable. For any $x_0 \in X$, the solution of the initial value problem
composed of \ref{e:ode} and the initial condition
$x(0)=x_0$ is extendable to $\mathbb{R}_+$.
\end{hypothesis}

\begin{theorem}
\label{thm:attsetC11}
Assume \reftohypothesis{\ref{h:continuoustime}}, let the right hand side $F$ of \ref{e:ode} be of class $C^{1,1}$, and let $\varphi$ denote the flow of \ref{e:ode}. Let $s,t>0$ and $\Omega \subseteq X$ be $s$-convex and closed with $\Omega \neq \mathbb{R}^n$. Further assume that there are $M_1,M_2, \lambda_-, \lambda_+ \in \mathbb{R}$ such that 
\begin{align}
M_1 &\geq 2 \mu_+( F'(x))-\mu_-(F'(x)),\label{eq:thm:attsetC11:1} \\ M_2&\geq \limsup_{h \to 0} \frac{\|F'(x+h)-F'(x)\|}{\|h\|}, \label{eq:thm:attsetC11:2} \\ 
 \lambda_- &\leq \mu_-(F'(x)) \leq \mu_+(F'(x)) \leq \lambda_+
 \label{eq:thm:attsetC11:3}
\end{align}
for all $x \in \varphi([0,t],\Omega)\subseteq X$, where $\mu_+(A)$ and $\mu_-(A)$ denote the maximum and minimum, respectively, eigenvalues of the symmetric part $(A+A^*)/2$ of $A$. If\linebreak$s M_2 \int_{0}^t \exp(M_1 \rho) d \rho <1$ then the attainable set $\varphi(\tau, \Omega)$ is $r$-convex for all $\tau \in [0,t]$ with
\begin{equation}
\label{e:thm:attset:C11:6}
r
=
\frac{s\exp((2 \lambda_+-\lambda_-)t)}{1-s M_2 \int_{0}^t \exp(M_1 \rho) d \rho}
\end{equation}
\end{theorem}
\begin{proof}
We may assume $\tau = t$ without loss of generality. By our hypothesis on the right hand side $F$ of \ref{e:ode}, the map $\Phi := \varphi(t, \cdot)$ is a $C^{1,1}$-diffeomorphism between an open neighborhood of $\Omega$ and an open subset of $X$. So, Corollary \ref{cor:diffeorad} can be applied to $\Phi$, and the required bounds \ref{eq:thm:diffeorad:1},\ref{eq:thm:diffeorad:2a} are obtained as follows. Let $D_2 \varphi$ denote the partial derivative of $\varphi$ with respect to the second argument. Let $h \in \R^n$ and $x\in \Omega$. The solution to the variational equation $y'(\rho) = F'(\varphi(\rho,x))y(\rho)$ of \ref{e:ode} along $\varphi(\cdot,x)$ is given by $y(\rho) = D_2 \varphi(\rho,x) h$, $\rho\in\intcc{0,t}$, and satisfies $y(0) = h$. Thus, the inequality $\|y(t)\| \leq \|h\| \exp( \int_{0}^t \mu_+(F'(\varphi(\rho,x))) d\rho)$ holds \cite{SansoneConti64}. Due to this fact, the bound $\|D_2\varphi(t,x)\| = \|\Phi'(x)\| \leq \exp({\lambda_+ t})$ is established using \ref{eq:thm:attsetC11:3}. By similar arguments one obtains $\|\Phi'(x)^{-1}\| \leq \exp({-\lambda_- t})$. See also \cite[\citeProposition~III.5]{i07MMAR}.\\
These bounds obviously lead to $\|\Phi'(x)\xi\|^2\cdot \|\Phi'(x)^{-1}\| \leq L_2\|\xi\|^2$ for $L_2 \defas \linebreak\exp((2\lambda_+ - \lambda_- )t)$, hence \ref{eq:thm:diffeorad:2a} holds for this choice of $L_2$. 
The bound \ref{eq:thm:diffeorad:1} is obtained by virtue of our hypotheses, \ref{eq:thm:attsetC11:1} and \ref{eq:thm:attsetC11:2}, in combination with arguments similar to the proof of \cite[\citeTheorem~IV.5]{i11abs}.
\end{proof}
\myvspace~\par
Theorem \ref{thm:attsetC11} can be applied quite easily in practice. Indeed, \ref{eq:thm:attsetC11:2} is an upper bound on the (local) Lipschitz constant of $F'$ (in case $F$ being of class $C^2$, it is a bound on $\| F''\|$) while a computation of the bounds \ref{eq:thm:attsetC11:1},\ref{eq:thm:attsetC11:3} for a particular example is done in \cite{i07Convex}. Theorem \ref{thm:attsetC2} below provides less conservative bounds for $r$ than Theorem \ref{thm:attsetC11} does as an analogue survey in case of convexity shows \cite{i07Convex}. 

\begin{theorem}
\label{thm:attsetC2}
Assume \reftohypothesis{\ref{h:continuoustime}}, let the right hand side $F$ of \ref{e:ode} be of class $C^2$. Let $\varphi,t,\Omega$ and $s$ be as in Theorem \ref{thm:attsetC11}, and let $\lambda_-,\lambda_+ \in \mathbb{R}$, such that \ref{eq:thm:attsetC11:3} is fulfilled for $F$. Let $D_2 \varphi(\tau,x)$ denote the partial derivative of $\varphi$ with respect to $x$. Assume further that there exists a constant $L_1\in \R_+$ such that 
\begin{equation}
\Bigg \| \int_{0}^\delta D_2\varphi(\tau,x)^{-1} F''(\varphi(\tau,x))(D_2\varphi(\tau,x)h)^2\, d \tau \Bigg \| \leq L_1 \|h\|^2
\label{e:thm:attsetC2:1}
\end{equation}
for all $x \in \Omega$, $\delta \in [0,t]$, and $h \in \mathbb{R}^n$. If $sL_1 <1$ then the attainable set $\varphi(\tau, \Omega)$ is $r$-convex for all $\tau \in [0,t]$ with
\begin{equation}
\label{e:thm:attsetC2:2}
r = \frac{s\exp((2 \lambda_+-\lambda_-)t)}{1-sL_1}.
\end{equation}
\end{theorem}
\begin{proof}
The proof is similar to that of Theorem \ref{thm:attsetC11}. We may assume $\tau = t$ without loss of generality. $\Phi \defas \varphi(t,\cdot)$ is $C^2$-diffeomorphism by assumption. We apply Corollary \ref{cor:diffeorad} to $\Phi$: In the case of $C^2$-smoothness, the right hand side of \ref{eq:thm:diffeorad:1} with $h$ at place of $\xi$ simplifies and therefore is bounded from above by $\|\Phi'(x)^{-1}\Phi''(x)h^2\|$. The integral in \ref{e:thm:attsetC2:1} with $t$ at place of $\delta$ equals $\Phi'(x)^{-1}\Phi''(x)h^2$ for $x \in X$, $h \in \R^n$ by the proof of \cite[\citeTheorem~IV.6]{i11abs}. Hence, \ref{eq:thm:diffeorad:1} holds by the bound in \ref{e:thm:attsetC2:1}. The bound in \ref{eq:thm:diffeorad:2a} is obtained in the same way as in Theorem \ref{thm:attsetC11}. 
\end{proof}
~\par
The bound in \ref{e:thm:attsetC2:2} is harder to verify than the one in \ref{e:thm:attset:C11:6} as the flow of \ref{e:ode} is explicitly involved in \ref{e:thm:attsetC2:1}. Nevertheless, the effort pays off since the accuracy of our novel approximation method, as detailed in the introduction, obviously depends on the radius of the balls used. Below, we illustrate the application of our novel method to abstraction based controller design. 

We consider a control system of the form $\dot x = F(x,u)$ under sampling where $u$ denotes a control that is constant on the half-open sampling intervals taking values $u_i \in \mathbb{R}$, $i \in \{1,\ldots,k\}$, $k \in \mathbb{N}$, and $F(\cdot,u) \colon X \subseteq \mathbb{R}^n \to \mathbb{R}$. Denote by $\varphi$ the general solution of the system $\dot x = F(x,u)$ which is defined by suitably extending the definition of the flow of \ref{e:ode}. We focus on the first step in abstraction based controller design, namely the computation of a discrete abstraction, and particularly, on the sets that are over-approximated in the computation. 

To obtain such an abstraction, the state space $X$ of the system is covered with polyhedra, so-called \begriff{cells}. In general, for a cell $\Delta \subseteq X$ the attainable set $\varphi(T,\Delta,u_i)$ is over-approximated for each $i$, where $T>0$ denotes the sampling time. (An approximation is required since the attainable set can be computed only numerically, in general.) The over-approximation set $\widehat \Omega \supseteq \varphi(T,\Delta,u_i)$ is then intersected with each cell of the covering. (An \textit{over}-approximation is required to ensure that all non-empty intersections due to the attainable set are recovered.) If a cell $\Delta_2$ has a non-empty intersection with $\widehat \Omega$, the attainable set $\varphi(T,\Delta_2 \cap \widehat \Omega,u_j)$ is over-approximated next for each $j\in \{1,\ldots,k\}$. Roughly speaking, information about non-empty intersections is stored suitably and this information determines the transitions in the discrete abstraction \cite{i11abs}.

Let us now discuss how to use our novel results for the required over-approximations. As a particular control system we consider a pendulum mounted on a cart where the acceleration $u \in \R$ of the cart is considered a control. The dynamics of the pole are given by\looseness=-1 
\begin{subequations}
\label{e:Pendulum}
\begin{align}
\dot x_1 &= x_2,\\
\dot x_2 &= -\omega^2 \sin(x_1) - u\ \omega^2 \cos(x_1) - 2 \gamma x_2
\end{align}
\end{subequations}
where $\omega > 0$ and $\gamma \geq 0$. We emphasize that $u$ is assumed to be constant.

From Theorem \ref{thm:attsetC2} we derive the following theorem about attainable sets of the particular system.
\begin{theorem}
\label{thm:Pendulum:Radius}
Let $t>0$ and assume the input $u$ in \ref{e:Pendulum} to be constant. Denote by $\varphi$ the general solution of \ref{e:Pendulum} (that is, $\varphi(\cdot,\cdot,u)$ denotes the flow of \ref{e:Pendulum}). Define 
\begin{align*}
\widehat \omega &= \max \left \{1, \omega \cdot(1+u^2)^{1/4} \right \}, \\
L_1(t) &= \frac{\operatorname{sinh}(3\widehat \omega t) + \operatorname{sinh}(\widehat \omega t)(12(\widehat \omega^{-2}+1)^{-3/2}-3)}{12\widehat \omega^2(1+(\widehat \omega + \gamma)^2)^{-3/2}}, \\ \lambda_\pm &= -\gamma \pm \sqrt{\gamma^2+(1+\widehat \omega^2)^2/4}.
\end{align*}
Assume $\gamma \leq 3\widehat \omega/4$ and $2 (\widehat \omega^2-\gamma^2)^{1/2}t \leq \pi$. Let $\Omega \subseteq \R^2$ be $s$-convex and closed. Then the attainable set $\varphi(t,\Omega,u)$ is $r(t)$-convex if $s\cdot L_1(t)<1$ where
\begin{align}
\label{e:radius}
r(t) &= \frac{s\operatorname{exp}((2\lambda_+-\lambda_-) t)}{1-sL_1(t)}.
\end{align} 
\end{theorem}
\begin{proof}
With $u$ being constant, \ref{e:Pendulum} is obviously of the form \ref{e:ode}. Having said this, we have 
\begin{equation*}
F'(x) = \begin{pmatrix} 0 & 1 \\ -\omega^2 \cos(x_1) + u\ \omega^2 \sin(x_1)  & -2\gamma \end{pmatrix},
\end{equation*}
and therefore we easily conclude that $\lambda_+$ and $\lambda_-$ are the maximum and minimum, respectively, eigenvalues of $(F'(x) + F'(x)^\ast)/2$. Thus, $\lambda_+$ and $\lambda_-$ satisfy \ref{eq:thm:attsetC11:3}. See also \cite[\citeSection~5]{i07Convex}. $L_1\defas L_1(t)$ satisfies
\ref{e:thm:attsetC2:1} \cite[\citeTheorem~A.4]{i11abs}, hence by Theorem
\ref{thm:attsetC2} in the present paper the proof is finished.
\end{proof}
\begin{example}
\label{ex:radius}
Let us consider system \ref{e:Pendulum} for $u \in \{-1,0\}$ with $\omega= 1$, $\gamma = 0.01$ and denote the general solution by $\varphi$. Suppose that $\Omega_1 \subseteq \R^2$ is a $0.4$-convex set of initial values. Let us determine the strong convexity of the attainable set $\varphi(T,\Omega_1,0)$ for sampling time $T = 0.32$. The numerical values imply the bounds 
\begin{equation*}
L_1(0.32) = \frac{\sinh(0.96) + \sinh(0.32) (6/\sqrt{2} - 3)}{12(1+1.01^2)^{-3/2}} \leq 0.37, \quad \lambda_+ \leq 1, \quad \lambda_-\geq -1.02,
\end{equation*}
on the constants required in Theorem \ref{thm:Pendulum:Radius}.
By \ref{e:radius} we conclude that $\varphi(T,\Omega_1,0)$ is $r$-convex for any $r\geq 1.24$. Thus, $\varphi(T,\Omega_1,0)$ can be over-approximated by supporting balls, practically performed as discussed in Section \ref{s:introduction}, with the radii of the balls being $1.24$. Similarly, the attainable set $\varphi(T,\widehat\Omega_{1},-1)$ of an $1.24$-convex set $\widehat\Omega_{1}$ can be over-approximated by supporting balls of radius $12$. A particular situation in the computation of a discrete abstraction for \ref{e:Pendulum} is illustrated in \ref{fig:attset}: $\Omega_1$ is a $0.4$-convex embedding of the cell $\Delta_1$. Therefore, $\varphi(T,\Omega_1,0)$ is over-approximated by four supporting balls leading to the approximation $\widehat\Omega_1$, which is a $1.24$-convex set by the calculation above. For $\Omega_{1,2}\defas \Omega_2 \cap \widehat\Omega_{1}$ the attainable set $\varphi(T,\Omega_{1,2},-1)$ is over-approximated using supporting balls of both radii $1.24$ and $12$ since $\varphi(T,\Omega_{1,2},-1)$ = $\varphi(T,\Omega_2,-1) \cap \varphi(T,\widehat\Omega_{1},-1)$, which results in the approximation $\widehat \Omega_{1,2}$. \\The advantage of using supporting balls lies in the fact that with the same data as required for the half-spaces the attainable sets are approximated more accurately. Consequently, fewer non-empty intersections due to conservative approximations occur, which increases the accuracy of the discrete abstraction (less spurious transitions): In \ref{fig:attset}, $\widehat \Omega_1 \cap \Delta_3 = \emptyset$ whereas if $\widehat \Omega_1$ had been defined by the supporting half-spaces the corresponding intersection would be non-empty. In \cite{i13qsuppc}, it is demonstrated that the abstraction based control design as briefly described in the present work benefits from this refined over-approximation method.
\end{example}
\definecolor{grau}{gray}{.5} 
\begin{figure}
\begin{center}
  \psfrag{1}[][]{$\Delta_1$}
\psfrag{2}[][]{$\Omega_1$}
\psfrag{3}[l][]{\footnotesize$\varphi(T,\Omega_1,0)$}
\psfrag{4}[][]{$\Omega_{1,2}$}
\psfrag{5}[][]{$\Omega_2$}
\psfrag{6}[l][]{\footnotesize$\varphi(T,\Omega_{1,2},-1)$}
\psfrag{7}[][]{\colorbox{white}{$\widehat \Omega_1$}}
\psfrag{8}[][]{$\widehat \Omega_{1,2}$}
\psfrag{9}[][]{\textcolor{grau}{\small$\Delta_3$}}
  \includegraphics[width=.58\linewidth]{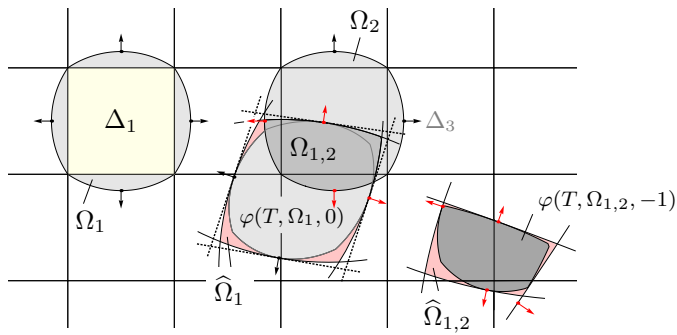}
\end{center}
\caption{\label{fig:attset} Illustration of Example \ref{ex:radius}.}
\vspace*{-\baselineskip}%
\end{figure}
\section{Conclusions}
\label{s:Conclusion}
We have developed necessary and sufficient conditions for connected 
sublevel sets of the form \ref{e:sublevelset} to be convex and strongly convex, 
respectively. The application to attainable sets of systems \ref{e:ode}
presented in Section \ref{s:application} has been the main motivation for our work.

Although it has been sufficient for our purposes to consider a finite 
dimensional setting, we emphasize that all of our results extend to 
arbitrary Hilbert spaces. In fact, the derivations in Section \ref{s:mainresultssublevelsets} do not 
need any modification if we only notice that the result in \cite{Rockafellar70} equally 
holds for infinite-dimensional cones that are given as the intersection of 
a finite number of half-spaces. Analogously, our results and arguments in 
Section \ref{s:application} remain valid under the additional assumption that the image 
$\Phi(\Omega)$ and the attainable set $\varphi(\tau,\Omega)$, 
respecticely, is closed, and for the attainable sets this is already 
guaranteed if the open set $X$ contains an $\varepsilon$-neighborhood of 
$\Omega$, for some positive $\varepsilon$. Only the proof of \cite[\citeLemma~A.2]{i11abs} needs to (and can) be adapted to the 
Hilbert space setting.

Among the issues that we leave open for future research are the relaxation 
of the constraint qualifications in Section \ref{s:mainresultssublevelsets} and the extension of our 
results on attainable sets in Section \ref{s:application} to differential inclusions under 
assumptions that are realistic from the control theory point of view.
\bibliographystyle{mysiam}
\bibliography{preambles,mrabbrev,strings,fremde,fremde2,eigeneJOURNALS,eigeneCONF}

\def\ocirc#1{\ifmmode\setbox0=\hbox{$#1$}\dimen0=\ht0 \advance\dimen0
  by1pt\rlap{\hbox to\wd0{\hss\raise\dimen0
  \hbox{\hskip.2em$\scriptscriptstyle\circ$}\hss}}#1\else {\accent"17 #1}\fi}
  \def\cprime{$'$} \ifx\hyperbaseurl\undefined
  \def\href#1#2{#2}\def\url#1{\texttt{#1}} \fi
  \ifx\ExplicitURLsInBibTeX\undefined\relax\else
  \def\href#1#2{www.reiszig.de/gunther/#1}\def\url#1{\texttt{#1}} \fi
  \ifx\hyperbaseurl\undefined\def\href#1#2{#2}\def\url#1{\texttt{#1}} \fi
  \ifx\hyperbaseurl\undefined\def\href#1#2{#2}\def\url#1{\texttt{#1}} \fi
\begin{thebibliography}{10}

\bibitem{AhmadiOlshevskyParriloTsitsiklis13}
{\sc A.~A. Ahmadi, A.~Olshevsky, P.~A. Parrilo, and J.~N. Tsitsiklis}, {\em
  N{P}-hardness of deciding convexity of quartic polynomials and related
  problems}, Math. Program., 137 (2013), pp.~453--476.

\bibitem{BednarikPastor04}
{\sc D.~Bedna{\v{r}}{\'{\i}}k and K.~Pastor}, {\em Elimination of strict
  convergence in optimization}, SIAM J. Control Optim., 43 (2004),
  pp.~1063--1077 (electronic).

\bibitem{BednarikPastor06}
{\sc D.~Bedna{\v{r}}{\'{\i}}k and K.~Pastor}, {\em Errata: ``{E}limination of
  strict convergence in optimization'' [{SIAM} {J}. {C}ontrol {O}ptim. {\bf 43}
  (2004), no. 3, 1063--1077 (electronic); 2114389]}, SIAM J. Control Optim., 45
  (2006), pp.~382--387 (electronic).

\bibitem{Blanc43}
{\sc E.~Blanc}, {\em Les ensembles surconvexes plans}, Ann. Sci. \'Ecole N.
  Sup. (3), 60 (1943), pp.~215--246.

\bibitem{CannarsaFrankowska06}
{\sc P.~Cannarsa and H.~Frankowska}, {\em Interior sphere property of
  attainable sets and time optimal control problems}, ESAIM Control Optim.
  Calc. Var., 12 (2006), pp.~350--370.

\bibitem{DemyanovRubinov95}
{\sc V.~F. Demyanov and A.~M. Rubinov}, {\em Constructive nonsmooth analysis},
  vol.~7 of Approximation \& Optimization, Peter Lang, Frankfurt am Main, 1995.

\bibitem{FattoriniFrankowska90}
{\sc H.~O. Fattorini and H.~Frankowska}, {\em Explicit convergence estimates
  for suboptimal controls. {I}}, Problems Control Inform. Theory, 19 (1990), pp.~3--29.

\bibitem{FattoriniFrankowska90b}
{\sc H.~O. Fattorini and H.~Frankowska}, {\em Explicit convergence estimates
  for suboptimal controls. {II}}, Problems Control Inform. Theory, 19 (1990), pp.~69--93.

\bibitem{FrankowskaOlech81}
{\sc H.~Frankowska and C.~Olech}, {\em {$R$}-convexity of the integral of
  set-valued functions}, in Contributions to analysis and geometry
  ({B}altimore, {M}d., U.S.A., 1980, Suppl. Amer. J. Math.), Johns Hopkins
  Univ. Press, Baltimore, Md., U.S.A, 1981, pp.~117--129.

\bibitem{GrueneJunge07}
{\sc L.~Gr{\"u}ne and O.~Junge}, {\em Approximately optimal nonlinear
  stabilization with preservation of the {L}yapunov function property}, in
  Proc. 46th IEEE Conf. Decision and Control (CDC), New Orleans, Louisiana,
  U.S.A., 2007, New York, 2007, IEEE, pp.~702--707.

\bibitem{GrueneMueller08}
{\sc L.~Gr{\"u}ne and F.~M{\"u}ller}, {\em Set oriented optimal control using
  past information}, in Proc. 2008 Math. Th. of Networks and Systems,
 Blacksburg, Virginia, U.S.A., July 28-Aug. 1, 2008.

\bibitem{HeltonNie09}
{\sc J.~W. Helton and J.~Nie}, {\em Sufficient and necessary conditions for
  semidefinite representability of convex hulls and sets}, SIAM J. Optim., 20
  (2009), pp.~759--791.

\bibitem{HeltonNie10}
{\sc J.~W. Helton and J.~Nie}, {\em Semidefinite representation of convex
  sets}, Math. Program., 122 (2010), pp.~21--64.

\bibitem{HenrionLouembet12}
{\sc D.~Henrion and C.~Louembet}, {\em Convex inner approximations of nonconvex
  semialgebraic sets applied to fixed-order controller design}, Internat. J.
  Control, 85 (2012), pp.~1083--1092.

\bibitem{HiriartUrrutyLemarechal93i}
{\sc J.-B. Hiriart-Urruty and C.~Lemar{\'e}chal}, {\em Convex analysis and
  minimization algorithms. {I}}, vol.~305 of Grundlehren der Math.
  Wissenschaften, Springer-Verlag, Berlin, 1993.

\bibitem{JourneeNesterovRichtarikSepulchre08}
{\sc M.~Journ{\'e}e, Y.~Nesterov, P.~Richt{\'a}rik, and R.~Sepulchre}, {\em
  Generalized power method for sparse principal component analysis}, J. Mach.
  Learn. Res., 11 (2010), pp.~517--553.

\bibitem{LaTorreRocca03}
{\sc D.~La~Torre and M.~Rocca}, {\em Remarks on second order generalized
  derivatives for differentiable functions with {L}ipschitzian {J}acobian},
  Appl. Math. E-Notes, 3 (2003), pp.~130--137.

\bibitem{Lasserre08}
{\sc J.~B. Lasserre}, {\em Convexity in semialgebraic geometry and polynomial
  optimization}, SIAM J. Optim., 19 (2008), pp.~1995--2014.

\bibitem{Lasserre09}
{\sc J.~B. Lasserre}, {\em Convex sets with semidefinite representation}, Math.
  Program., 120 (2009), pp.~457--477.

\bibitem{Lasserre10}
{\sc J.~B. Lasserre}, {\em Certificates of convexity for basic semi-algebraic
  sets}, Appl. Math. Lett., 23 (2010), pp.~912--916.

\bibitem{LevitinPolyak66}
{\sc E.~S. Levitin and B.~T. Polyak}, {\em Minimization methods in the presence
  of constraints}, {\v{Z}}. Vy{\v{c}}isl. Mat. i Mat. Fiz., 6 (1966),
  pp.~787--823.
\newblock (Russian. Engl. transl. in U.S.S.R. Comput. Math. and Math. Phys.,
  vol. 6, no. 5, 1966, 1-50).

\bibitem{LiXu10}
{\sc S.~J. Li and S.~Xu}, {\em Sufficient conditions of isolated minimizers for
  constrained programming problems}, Numer. Funct. Anal. Optim., 31 (2010),
  pp.~715--727.


\bibitem{Mayer35}
{\sc A.~E. Mayer}, {\em Eine \"{U}berkonvexit\"at}, Math. Z., 39 (1935),
  pp.~511--531.

\bibitem{Plis74}
{\sc A.~Pli{\'s}}, {\em Accessible sets in control theory}, in Proc. Int. Conf.
  Diff. Equations (Univ. Southern California, Los Angeles, CA, U.S.A., Sept.
  3--7, 1974), H.~A. Antosiewicz, ed., Academic Press, New York, 1975,
  pp.~646--650.

\bibitem{Polovinkin96}
{\sc E.~S. Polovinkin}, {\em Strongly convex analysis}, Mat. Sb., 187 (1996),
  pp.~103--130.
\newblock (Russian. Engl. transl. in Russian Acad. Sci. Sb. Math., vol. 187,
  1996, no. 2, 259--286).

\bibitem{i07MMAR}
{\sc G.~Rei{\ss}ig}, {\em Convexity of reachable sets of nonlinear
  discrete-time systems}, in Proc. 13th IEEE Int. Conf. Methods and Models in
  Automation and Robotics (\nobreak{MMAR}), Szczecin, Poland, Aug. 27-30, 2007,
  R.~Kaszy{\'n}ski, ed., 2007, pp.~199--204.

\bibitem{i07Convex}
{\sc G.~Rei{\ss}ig}, {\em Convexity of reachable sets of nonlinear ordinary
  differential equations}, Automat. Remote Control, 68 (2007), pp.~1527--1543.
\newblock \href{http://arxiv.org/abs/1211.6080}{arXiv:1211.6080}

\bibitem{i11abs}
{\sc G.~Rei{\ss}ig}, {\em Computing abstractions of nonlinear systems}, IEEE
  Trans. Automat. Control, 56 (2011), pp.~2583--2598.
\newblock \href{http://arxiv.org/abs/0910.2187}{arXiv:0910.2187}.

\bibitem{ReissigRungger13}
{\sc G.~Reissig and M.~Rungger}, {\em Abstraction-based solution of optimal
  stopping problems under uncertainty}, in Proc. 52nd IEEE Conf. Decision and
  Control (CDC), Florence, Italy, 10-13 Dec. 2013, New York, 2013, IEEE,
  pp.~3190--3196.

\bibitem{Rockafellar70}
{\sc R.~T. Rockafellar}, {\em Convex analysis}, Princeton Mathematical Series,
  No. 28, Princeton University Press, Princeton, N.J., U.S.A., 1970.

\bibitem{RunggerMazoTabuada13}
{\sc M.~Rungger, M.~Mazo, Jr., and P.~Tabuada}, {\em Specification-guided
  controller synthesis for linear systems and safe linear-time temporal logic},
  in Proceedings of the 16th international conference on Hybrid systems:
  computation and control, HSCC '13, New York, NY, USA, 2013, ACM,
  pp.~333--342.

\bibitem{SansoneConti64}
{\sc G.~Sansone and R.~Conti}, {\em Non-linear differential equations}, vol.~67
  of International Series of Monographs in Pure and Applied Mathematics, A
  Pergamon Press Book. The Macmillan Co., New York, revised~ed., 1964.
\newblock Transl. from the Italian by Ainsley H. Diamond.

\bibitem{Tabuada09}
{\sc P.~Tabuada}, {\em Verification and control of hybrid systems}, Springer,
  New York, 2009.
\newblock A symbolic approach, With a foreword by Rajeev Alur.

\bibitem{Valentine64}
{\sc F.~A. Valentine}, {\em Convex sets}, McGraw-Hill Series in Higher
  Mathematics, McGraw-Hill Book Co., New York, 1964.
\newblock {G}erman transl.: "Konvexe {M}engen", Bibl. Inst., Mannheim, 1968.

\bibitem{Veliov89}
{\sc V.~Veliov}, {\em Second order discrete approximations to strongly convex
  differential inclusions}, Systems Control Lett., 13 (1989), pp.~263--269.

\bibitem{Vial83}
{\sc J.-P. Vial}, {\em Strong and weak convexity of sets and functions}, Math.
  Oper. Res., 8 (1983), pp.~231--259.

\bibitem{Walter04}
{\sc W.~Walter}, {\em Analysis 1:}, Analysis, Springer, 2004.

\bibitem{i12sConvex}
{\sc A.~Weber and G.~Rei{\ss}ig}, {\em Local characterization of strongly
  convex sets}, J. Math. Anal. Appl., 400 (2013), pp.~743--750.
\newblock \href{http://arxiv.org/abs/1207.4347}{arXiv:1207.4347}.

\bibitem{i13qsuppc}
{\sc A.~Weber and G.~Reissig}, {\em Strongly convex attainable sets and low
  complexity finite-state controllers}, in Proc. Australian Control Conf.
  (AUCC), Perth, Australia, 4-5 Nov. 2013, 2013, pp.~61--66.

\bibitem{Zeidler.i}
{\sc E.~Zeidler}, {\em Fixed--{P}oint {T}heorems}, vol.~I of Nonlinear
  Functional Analysis and its Applications, Springer, corrected~ed., 1993.

\end{thebibliography}
\end{document}